\documentclass[11pt,letterpaper,reqno]{amsart}

\usepackage{amssymb,amsmath,amsthm,amsfonts}
\usepackage{bbm}
\usepackage{enumitem}
\usepackage{booktabs}
\usepackage{graphicx}
\usepackage[T1]{fontenc}
\usepackage{doi}
\usepackage{tabularx}
\usepackage{booktabs}
\usepackage{array}
\usepackage{float}
\usepackage{bookmark}
\usepackage{hyperref}
\hypersetup{pdfstartview={FitH}}

\newtheorem{thm}{Theorem}[section]
\newtheorem{lem}[thm]{Lemma}
\newtheorem{prop}[thm]{Proposition}
\newtheorem{cor}[thm]{Corollary}
\newtheorem{cla}[thm]{Claim}

\newtheorem{ques}[thm]{Question}

\theoremstyle{definition}

\newtheorem{rem}[thm]{Remark}

\numberwithin{equation}{section}

\newcommand{\B}{\mathcal{B}}
\newcommand{\Hh}{\mathcal{H}}
\newcommand{\Kk}{\mathcal{K}}
\newcommand{\Id}{I}
\newcommand{\Mn}{\mathbb{M}_n}
\newcommand{\Ree}{\operatorname{Re}}
\newcommand{\Ima}{\operatorname{Im}}

\newcommand{\absop}[1]{\left|#1\right|}
\newcommand{\sym}[1]{\left|#1\right|_{\mathrm{sym}}}
\newcommand{\qsym}[1]{\left|#1\right|_{\mathrm{qsym}}}
\newcommand{\rank}{\operatorname{rank}}

\newcommand{\tr}{\operatorname{Tr}\,}

\newcommand{\Ran}{\operatorname{Ran}}

\DeclareMathOperator{\ran}{Ran}
\DeclareMathOperator{\supp}{supp}
\DeclareMathOperator{\diag}{diag}
\DeclareMathOperator{\Tr}{Tr}
\newcommand{\unorm}[1]{%
	\left|\mkern-1.5mu\left|\mkern-1.5mu\left|#1\right|
	\mkern-1.5mu\right|\mkern-1.5mu\right|%
}
\begin{document}
	
\title[Triangle inequality for the  quadratic symmetric modulus]{An operator triangle inequality for the  quadratic symmetric modulus}
	
	\author[T.~Zhang]{Teng Zhang}
	\address{School of Mathematics and Statistics, Xi'an Jiaotong University, Xi'an 710049, P. R. China}
	\email{teng.zhang@stu.xjtu.edu.cn}
	
\subjclass[2020]{Primary 47A63; Secondary 47A30, 47A05}
\keywords{Quadratic symmetric modulus, triangle inequality, unitary orbits, equality cases, Clarkson--McCarthy type inequalities, Euler's identity}
	
	\begin{abstract}
50 years after Thompson's famous triangle inequality for the operator right modulus, we establish a triangle inequality for the quadratic symmetric modulus. We also discuss the corresponding equality cases as well as the infinite-dimensional setting. In addition, we obtain Clarkson--McCarthy type inequalities for the quadratic symmetric modulus.  Moreover, we answer several questions raised by Bourin and Lee in [\emph{Bull. Lond. Math. Soc.} \textbf{44} (2012), no.~6, 1085--1102] and [\emph{Internat. J. Math.} \textbf{31} (2026), no.~6, 2650018].
	\end{abstract}
	
	\maketitle
	
 \tableofcontents
	
	\section{Introduction}\label{sec:introduction}
	Let \(\mathbb{M}_{m,n}\) denote the space of all \(m\times n\) complex matrices, and write \(\mathbb{M}_n:=\mathbb{M}_{n,n}\). Let \(\B(\Hh,\Kk)\) denote the space of all bounded linear operators from a complex Hilbert space \(\Hh\) to a complex Hilbert space \(\Kk\); when \(\Hh=\Kk\), we simply write \(\B(\Hh)\).  We denote by \(\mathbb K(\Hh)\) the ideal of compact operators on \(\Hh\); when the underlying Hilbert space is clear, we simply write \(\mathbb K\).
	
\subsection{Thompson's inequality and its development}
In 1976,	 Thompson \cite{Tho76,Tho77} obtained his famous  triangle inequality for the operator right modulus.
	
	\begin{thm}[Thompson]\label{thm:thompson}
		Let $A,B\in\Mn$. Then there exist two unitaries $U,V$ such that
		\[
		|A+B|\le U|A|U^*+V|B|V^*.
		\]
	\end{thm}
	Thompson \cite{Tho79} also characterized the equality case in Theorem~\ref{thm:thompson}.
	
	\begin{thm}[Thompson]\label{thm:thompson-2}
		Let $A,B\in\Mn$. The following are equivalent:
		\begin{enumerate}[label=(\roman*)]
			\item there exist unitaries $U,V\in\Mn$ such that
			\[
			|A+B|= U|A|U^*+V|B|V^*;
			\]
			\item there exists a unitary $W\in\Mn$ such that
			\[
			A=W|A|,\qquad B=W|B|.
			\]
		\end{enumerate}
	\end{thm}

	Passing from matrices to operators on an infinite-dimensional Hilbert space, the triangle inequality
	persists with unitaries replaced by isometries thanks to Akemann, Anderson and Pedersen \cite{AAP82}:
	
	\begin{thm}[Akemann--Anderson--Pedersen]\label{thm:AAP}
		Let $A,B\in\B(\Hh)$. Then there exist two isometries $U,V$ such that
		\[
		|A+B|\le U|A|U^*+V|B|V^*.
		\]
	\end{thm}
		Naturally, we also ask when equality can occur in Theorem~\ref{thm:AAP}.
	However, the equality characterization in Theorem~\ref{thm:thompson-2} is genuinely finite-dimensional.
	Indeed, Ando and Hayashi \cite{AH07} remarked that Theorem~\ref{thm:thompson-2} fails in the
	infinite-dimensional setting: one can find $A,B\in\B(\Hh)$ and unitaries $U,V$ such that
	\[
	|A+B| = U|A|U^{*} + V|B|V^{*},
	\]
	while there is no partial isometry $W$ satisfying simultaneously
	\[
	A= W|A| \quad \text{and} \quad B = W|B|.
	\]
	For instance, let $P$ be an orthogonal projection and let $Q$ be a subprojection of $P$.
	Assume that $P$, $Q$, and $P-Q$ all have infinite rank. Then there exist (at least countably many)
	unitaries $U,V$ for which
	\[
	|P+(-Q)| = U|P|U^{*} + V|{-}Q|V^{*}
	\]
	holds, whereas there is no partial isometry $W$ such that
	\[
	P = W|P| \quad \text{and} \quad (-Q) = W|{-}Q|.
	\]
	
	Even though Theorem~\ref{thm:thompson-2} breaks down in infinite dimension, Ando and Hayashi \cite{AH07}
	showed that the \emph{intrinsic triangle equality} for the usual modulus still admits a clean description.
\begin{thm}[Ando--Hayashi]\label{thm:AH-classical}
	Let $A,B\in\B(\Hh)$.
	If
	\[
	\absop{A+B}=\absop{A}+\absop{B},
	\]
	then there exists a partial isometry $U\in\B(\Hh)$ such that
	\[
	A=U\absop{A},\qquad B=U\absop{B}.
	\]
\end{thm}

Aujla and Bourin \cite{AB07} established a matrix subadditivity inequality for concave functions.
\begin{thm}[Aujla--Uchiyama]\label{thm:Bourin--Uchiyama}
	Let $A,B\ge0$ and $f:[0,\infty)\to \mathbb{R}$ be monotone  concave with $f(0)\ge 0$. Then there exist unitaries $U,V$ such that
	\[
	f(A+B)\le Uf(A)U^*+Vf(B)V^*.
	\]
\end{thm}

Combining Theorems~\ref{thm:thompson} and~\ref{thm:Bourin--Uchiyama} with the standard fact that, whenever $0\le X\le Y$ and $f$ is nondecreasing,
\[
\lambda_j\bigl(f(X)\bigr)\le \lambda_j\bigl(f(Y)\bigr), \qquad 1\le j\le n,
\]
where $\lambda_j(\cdot)$ denotes the $j$-th largest eigenvalue, and hence
\[
f(X)\le Wf(Y)W^*
\]
for some unitary $W$, we obtain the following: 
for every monotone concave function \(f:[0,\infty)\to\mathbb R\) with \(f(0)\ge0\), there exist unitaries
\(U,V\in\mathbb M_n\) such that
\begin{equation}\label{eq:triangle-concave}
	f(|A+B|)\le Uf(|A|)U^*+Vf(|B|)V^*,
	\qquad A,B\in\mathbb M_n,
\end{equation}
see \cite{BU07}.
Thus Thompson's inequality extends from the identity function \(f(t)=t\) to all monotone concave
functions with \(f(0)\ge0\).

Since \(f\) is monotone and \(f(0)\ge0\), we automatically have \(f(t)\ge0\) for all \(t\ge0\).
Conversely, every concave map \(f:[0,\infty)\to[0,\infty)\) is automatically nondecreasing.
Indeed, if \(f\) were not nondecreasing, then there would exist \(0\le s<t\) such that \(f(s)>f(t)\), and hence
\[
\frac{f(t)-f(s)}{t-s}<0.
\]
By concavity, for every \(x>t\),
\[
\frac{f(x)-f(t)}{x-t}\le \frac{f(t)-f(s)}{t-s}<0,
\]
so that
\[
f(x)\le f(t)+\frac{x-t}{t-s}\bigl(f(t)-f(s)\bigr).
\]
Since \(f(t)-f(s)<0\), the right-hand side tends to \(-\infty\) as \(x\to\infty\), contradicting the
assumption that \(f(x)\ge0\) for all \(x\ge0\). Therefore every concave map
\(f:[0,\infty)\to[0,\infty)\) is nondecreasing.

This observation shows that, for concave functions on \([0,\infty)\), nonnegativity of the range is
equivalent to monotonicity. Bourin and Lee \cite[Remark~3.13]{BL12}, and later Bourin, Harada, and Lee
\cite[Question~1.3]{BHL14}, asked whether the monotonicity assumption in
Theorem~\ref{thm:Bourin--Uchiyama} can be removed if one only assumes concavity together with \(f(0)\ge0\).

\begin{ques}[Bourin--Harada--Lee]\label{ques:Bourin-Harada-Lee}
	It is not known whether the monotonicity assumption in Theorem~\ref{thm:Bourin--Uchiyama} can be deleted,
	that is, whether the theorem remains valid for every concave function \(f:[0,\infty)\to\mathbb R\) satisfying
	\(f(0)\ge0\).
\end{ques}

We shall show that the monotonicity assumption in Theorem~\ref{thm:Bourin--Uchiyama} is indispensable.
More precisely, we construct a counterexample using the concave function \(f(t)=t-t^2\), thereby answering
Question~\ref{ques:Bourin-Harada-Lee} in the negative.

For further developments regarding the technique of unitary orbits, we refer the reader to a good survey \cite{BL12}.
\subsection{A triangle inequality for the quadratic symmetric modulus}
To obtain new triangle inequalities for the newly introduced operator moduli, let us first recall the definition of the modulus in the scalar case. For a complex number \(z=\Re z+i\,\Im z\in\mathbb C\), the modulus is classically defined by
\begin{equation}\label{eq:z-def-1}
	|z|:=\sqrt{(\Re z)^2+(\Im z)^2},
\end{equation}
which immediately yields
\begin{equation}\label{eq:z-def-2}
	|z|=(\bar z z)^{\frac{1}{2}}
\end{equation}
and
\begin{equation}\label{eq:z-def-3}
	|z|=\frac{|z|+|\bar z|}{2}.
\end{equation}

 For \(Z\in\mathbb{M}_n\), the \emph{right modulus}, in analogy with \eqref{eq:z-def-2}, is defined by
\[
|Z| := (Z^*Z)^{\frac{1}{2}},
\]
where \(Z^*\) denotes the conjugate transpose of \(Z\).
On the other hand, if we write the Cartesian decomposition \(Z=\Re Z+i\,\Im Z\), where
\[
\Re Z := \frac{Z+Z^*}{2},\qquad \Im Z := \frac{Z-Z^*}{2i},
\]
then a more geometric analogue of \eqref{eq:z-def-1} motivates the definition of the \emph{quadratic symmetric modulus}:
\[
|Z|_{\mathrm{qsym}} := \sqrt{(\Re Z)^2+(\Im Z)^2}
= \left(\frac{|Z|^2+|Z^*|^2}{2}\right)^{\frac{1}{2}}.
\]
We also define the \emph{arithmetic symmetric modulus}, following \eqref{eq:z-def-3}, by
\[
|Z|_{\mathrm{sym}} := \frac{|Z|+|Z^*|}{2}.
\]

We now turn to a related problem raised by Bourin and Lee concerning Thompson-type triangle inequalities for symmetric moduli, in comparison with Theorem~\ref{thm:thompson}.
\begin{ques}[{\cite[Question~4.9]{BL26}}]\label{ques:BL-Q49}
	Let \(|Z|_{\mathrm{(q)sym}}\) denote either the arithmetic symmetric modulus or the quadratic symmetric modulus of \(Z\in\mathbb{M}_n\). For \(A,B\in \mathbb{M}_n\), does there exist a pair of unitaries \(U,V\in \mathbb{M}_n\) such that the following Thompson-type triangle inequality holds?
	\begin{equation}\label{eq:sym_qsym}
		|A+B|_{\mathrm{(q)sym}}
		\le
		U\,|A|_{\mathrm{(q)sym}}\,U^{*}+V\,|B|_{\mathrm{(q)sym}}\,V^{*}.
	\end{equation}
\end{ques}
Our next two theorems give a complete answer to Question~\ref{ques:BL-Q49}: the inequality
\eqref{eq:sym_qsym} fails for the arithmetic symmetric modulus, whereas it does hold for the quadratic
symmetric modulus.

\begin{thm}\label{thm:non-exist-sym}
	Let $A,B\in \mathbb{M}_n$.
		In general, there do not exist unitaries $U,V$ such that
		\[
		|A+B|_{\mathrm{sym}}\le U|A|_{\mathrm{sym}}U^*+V|B|_{\mathrm{sym}}V^* .
		\] Indeed, there exist $A,B\in\mathbb{M}_2$ such that
		\[
		\bigl\||A+B|_{\mathrm{sym}}\bigr\|_{\infty}
		>
		\bigl\||A|_{\mathrm{sym}}\bigr\|_{\infty}+\bigl\||B|_{\mathrm{sym}}\bigr\|_{\infty}.
		\]
\end{thm}
	
	\begin{thm}\label{thm:qsym-thompson-matrix}
		Let $A,B\in\Mn$. Then there exist two unitaries $U,V$ such that
		\[
		|A+B|_{\mathrm{qsym}}\le U|A|_{\mathrm{qsym}}U^*+V|B|_{\mathrm{qsym}}V^*.
		\]
	\end{thm}
We remark that Theorem~\ref{thm:qsym-thompson-matrix} can be regarded as an operator version of the classical triangle inequality in the complex plane,
	$|a+b|\le |a|+|b|$ for $a,b\in\mathbb{C}$, a result of historical importance.
	In contrast to Theorem~\ref{thm:thompson}, here the modulus is defined via the real and imaginary parts of the \emph{operator itself}. 
	In this sense, it may be viewed as a more faithful operator analogue of the triangle inequality.
	
	We also characterize the equality case in Theorem~\ref{thm:qsym-thompson-matrix} as follows.
	\begin{thm}\label{thm:equality-case-matrix}
		Let $A,B\in\Mn$. Assume that there exist unitaries $U,V\in\Mn$ such that
		\[
		\qsym{A+B}=U\,\qsym{A}\,U^*+V\,\qsym{B}\,V^*.
		\]
		Define the block matrix
		\[
		\Gamma(T):=\frac{1}{\sqrt2}\begin{pmatrix}T\\ T^*\end{pmatrix}\in \mathbb{M}_{2n,n}.
		\]
		Then there exists a partial isometry $W\in\mathbb{M}_{2n,n}$ such that
		\[
		\Gamma(A)=W\,\qsym{A},
		\qquad
		\Gamma(B)=W\,\qsym{B}.
		\]
		Equivalently, $\Gamma(A)$ and $\Gamma(B)$ admit polar decompositions sharing a common partial isometry factor.
	\end{thm}
	
	The next question is whether Theorem~\ref{thm:qsym-thompson-matrix} remains true in the infinite-dimensional setting. Our first observation is that it does: the quadratic symmetric Thompson type inequality extends to
	$\B(\Hh)$, provided one replaces unitaries by isometries, exactly as in Theorem~\ref{thm:AAP}.
	
	\begin{thm}\label{thm:qsym-thompson-operator}
		Let $A,B\in\B(\Hh)$. Then there exist isometries $U,V\in\B(\Hh)$ such that
		\[
		\qsym{A+B}\ \le\ U\,\qsym{A}\,U^* \;+\; V\,\qsym{B}\,V^* .
		\]
	\end{thm}

	Having established the triangle inequality, we turn to its equality case.
	Motivated by Theorem~\ref{thm:AH-classical}, we  prove an
	Ando--Hayashi type characterization for the quadratic symmetric modulus.
	Since $\qsym{T}$ is not, in general, the modulus of $T$ itself, the correct formulation uses a canonical lifting.
	
	\begin{thm}\label{thm:main-equality}
		Let $A,B\in\B(\Hh)$. Assume that
		\begin{equation*}
			\qsym{A+B}=\qsym{A}+\qsym{B}.
		\end{equation*}
		Define the linear lifting $\Gamma:\B(\Hh)\to\B(\Hh,\Hh\oplus\Hh)$ by
		\begin{equation*}
			\Gamma(T):=\frac{1}{\sqrt2}\binom{T}{T^*}.
		\end{equation*}
		Then there exists a partial isometry $W\in\B(\Hh,\Hh\oplus\Hh)$ such that
		\[
		\Gamma(A)=W\,\qsym{A},
		\qquad
		\Gamma(B)=W\,\qsym{B}.
		\]
		Equivalently, $\Gamma(A)$ and $\Gamma(B)$ admit polar decompositions sharing a common partial isometry factor.
	\end{thm}
Besides triangle inequalities, we also establish a concave-function subadditivity theorem; compare this with \eqref{eq:triangle-concave}.
	\begin{thm}\label{thm:intro-main-concave}
		Let $A,B\in\mathbb M_n$ and let $f:[0,\infty)\to[0,\infty)$ be  concave. Then there exist unitaries $U,V\in\mathbb M_n$ such that
		\[
		f\bigl(\qsym{A+B}\bigr)\le Uf\bigl(\qsym{A}\bigr)U^*+Vf\bigl(\qsym{B}\bigr)V^*.
		\]
	\end{thm}
A number of elegant operator inequalities can be derived from Theorem~\ref{thm:intro-main-concave}. For example,	taking $f(t)=t^p$ with $0<p<1$ in Theorem~\ref{thm:intro-main-concave} yields
	\[
	\qsym{A+B}^p\le U\qsym{A}^pU^*+V\qsym{B}^pV^*,\qquad 0<p<1.
	\]
	
Let $\unorm{\cdot}$ denote a unitarily invariant norm on $\Mn$. An immediate corollary of Theorem~\ref{thm:intro-main-concave} is as follows.
	\begin{cor}
		Let $A,B\in\mathbb M_n$ and let $f:[0,\infty)\to[0,\infty)$ be  concave. Then for any unitarily invariant norm, we have
		\[
		\unorm{f\bigl(\qsym{A+B}\bigr)}	\le  \unorm{f\bigl(\qsym{A}\bigr)}+ \unorm{f\bigl(\qsym{B}\bigr)}.
		\]
	\end{cor}
	Taking traces in Theorem~\ref{thm:intro-main-concave} yields a Rotfel'd type trace inequality for  the  quadratic symmetric modulus.
	\begin{cor}
		Let $A,B\in\mathbb M_n$ and let $f:[0,\infty)\to[0,\infty)$ be  concave. Then 
		\[
		\tr	f\bigl(\qsym{A+B}\bigr)\le \tr f\bigl(\qsym{A}\bigr)+\tr f\bigl(\qsym{B}\bigr).
		\]
	\end{cor}

We can also derive the following corollary, which may be regarded as an operator version of the scalar
inequality \(|a|-|b|\le |a-b|\) for \(a,b\in\mathbb C\).
	\begin{cor}\label{cor:triangle-plus}
		Let $A,B\in\mathbb M_n$ and let $f:[0,\infty)\to[0,\infty)$ be  concave. 
		Then there exist unitaries $U,V\in\mathbb M_n$ such that
		\[
		Uf\bigl(\qsym{A}\bigr)U^*-Vf\bigl(\qsym{B}\bigr)V^*
		\le
		f\bigl(\qsym{A-B}\bigr).
		\]
	\end{cor}
	
If \(A\in \mathbb{M}_n\) is invertible, its condition number is defined by
\[
\kappa(A):=\|A\|_\infty \|A^{-1}\|_\infty,
\]
where \(\|\cdot\|_\infty\) denotes the spectral norm. Lee \cite{Lee12} proved the following triangle inequality: if \(A_1,\dots,A_m\in \mathbb{M}_n\) are invertible matrices whose condition numbers are bounded by \(\omega>0\), then
\begin{equation}\label{eq:Lee}
	\absop{A_1+\cdots+A_m}
	\le
	\frac{\omega+1}{2\sqrt{\omega}}\,
	\bigl(\absop{A_1}+\cdots+\absop{A_m}\bigr).
\end{equation}
Bourin and Lee \cite[p.~1096, line 4]{BL12} remarked that
 the bound \eqref{eq:Lee} is independent of the number of operators, although it is a rather small bound, it is not known whether it is sharp, see also \cite[Question~2]{Lee12}. In this paper, we will give an example to illustrate its sharpness. Moreover, we obtain the following more refined versions of \eqref{eq:Lee} in terms of the three introduced moduli.
 \begin{thm}\label{thm:triangle-another}
 	Let \(A_1,\dots,A_m\in \B(\Hh)\) be invertible operators whose condition numbers are $\kappa_1,\ldots,\kappa_m$.
 	Then
 	\[
 	\bigl|A_1+\cdots+A_m\bigr|_{\mathrm{((q)sym)}}
 	\le
 	\frac{\kappa_1+1}{2\sqrt{\kappa_1}}\,
 	|A_1|_{\mathrm{((q)sym)}}+\cdots+ 	\frac{\kappa_m+1}{2\sqrt{\kappa_m}}\,|A_m|_{\mathrm{((q)sym)}},
 	\]
where the coefficient function
$
c(\kappa):=\frac{\kappa+1}{2\sqrt{\kappa}}
$
is individually sharp for each term: for every \(\kappa\ge 1\), the coefficient \(c(\kappa)\)
in front of a term of condition number \(\kappa\) cannot be replaced by any smaller constant,
even if all the other coefficients are kept unchanged.
 \end{thm}

\subsection{Clarkson--McCarthy type inequalities for the quadratic symmetric modulus} Let \(\|\cdot\|_p\) denote the Schatten \(p\)-norm for \(1 \le p \le \infty\), and the quasi-Schatten \(p\)-norm for \(0 < p < 1\).
The classical Clarkson--McCarthy inequalities \cite{McC67} state that for $A,B\in \Mn$ and $p\ge 2$,
\begin{equation}\label{eq:Mccarthy}
	\|A+B\|_p^p+\|A-B\|_p^p\le 2^{p-1}\bigl(\|A\|_p^p+\|B\|_p^p\bigr).
\end{equation}
For $0<p\le 2$, the inequality is reversed. In particular, \eqref{eq:Mccarthy} shows that the space $\Mn$ is $p$-uniformly convex for $p\ge 2$. Bourin and Lee \cite{BL20} established the following  unitary-orbit refinement of \eqref{eq:Mccarthy}: for  $p\ge 2$, there exist unitaries $U,V$ such that
\begin{equation}\label{eq:Clarkson-unitary-orbit}
	U|A+B|^pU^*+V|A-B|^pV^*\le 2^{p-1}\bigl(|A|^p+|B|^p\bigr).
\end{equation}
For $0<p\le 2$, the inequality is reversed. The Clarkson--McCarthy inequalities have since been extended in various directions \cite{BH88}, including to multivariable inequalities involving several operators \cite{BK04,Zha25}; for a unified framework, see the author's recent paper \cite{Zha24}.

We now turn from triangle inequalities to Clarkson--McCarthy type estimates for quadratic symmetric modulus.
The starting point is an exact square-sum identity, which plays the role of a Parseval formula for the quadratic symmetric modulus.
\begin{thm}\label{thm:identity}
	For a fixed $m\times n$ matrix $U=(u_{ij})$,
	the operator square-sum identity
	\begin{equation}\label{eq:identity}
		\sum_{i=1}^m
		\left|
		\sum_{j=1}^n u_{ij}A_j
		\right|_{\mathrm{qsym}}^{2}
		=
		\sum_{j=1}^n |A_j|_{\mathrm{qsym}}^{2}
	\end{equation}
	holds for all $A_1,\ldots,A_n\in \mathbb{M}_N$
	if and only if $U$ is an isometry, that is,
	$
	U^*U=I_n.
	$
\end{thm}
With the square-sum identity in hand, we may combine it with Jensen-type orbit inequalities to derive the desired Clarkson--McCarthy type estimates.
\begin{thm}\label{thm:unitary-orbit}
	Let $B_i, A_j \in \mathbb M_N$ for $1\le i\le m, 1\le j\le n$ satisfy $(B_1,\ldots,B_m)^T=U(A_1,\ldots,A_n)^T$ for some isometry $U=(u_{ij})$. Then for $p\ge 2$, there exist unitaries $U_i,V_j$ such that
\begin{align*}
			\sum_{i=1}^mU_i\left|B_i\right|_{\mathrm{qsym}}^pU_i^*&\le n^{\frac{p}{2}-1}\sum_{j=1}^n\left|A_j\right|_{\mathrm{qsym}}^p,\\
	\sum_{j=1}^nV_j\left|A_j\right|_{\mathrm{qsym}}^pV_j^*&\le m^{\frac{p}{2}-1}\sum_{i=1}^m\left|B_i\right|_{\mathrm{qsym}}^p.
\end{align*}
	For $0 < p \le 2$, the inequalities are reversed.
\end{thm}
By choosing the normalized Hadamard isometry
\[
\frac{1}{\sqrt{2}}
\begin{pmatrix}
	1&1\\
	1&-1
\end{pmatrix}
\]
in Theorem~\ref{thm:unitary-orbit} yields the following two-variable Clarkson--McCarthy type inequality.
	\begin{cor}
	Let $A,B\in \mathbb M_n$. Then for $p\ge 2$, there exist unitaries $U,V$ such that
	\[
	U\qsym{A+B}^pU^*+V\qsym{A-B}^pV^*\le 2^{p-1}(\qsym{A}^p+\qsym{B}^p).
	\]
	For $0<p\le 2$, the inequality is reversed.
\end{cor}
We next combine Theorem~\ref{thm:identity} with certain Jensen-type inequalities for unitarily invariant norms
to derive the following Clarkson--McCarthy type estimates.
\begin{thm}\label{thm:uin}
		Let $B_i, A_j \in \mathbb M_N$ for $1\le i\le m, 1\le j\le n$ satisfy $(B_1,\ldots,B_m)^T=U(A_1,\ldots,A_n)^T$ for some isometry $U=(u_{ij})$. Then for $p\ge 2$,
	\begin{align*}
m^{1-\frac{p}{2}}\unorm{ \sum_{j=1}^n\left|A_j\right|_{\mathrm{qsym}}^p} \le	\unorm{	\sum_{i=1}^m\left|B_i\right|_{\mathrm{qsym}}^p} &\le n^{\frac{p}{2}-1}\unorm{\sum_{j=1}^n\left|A_j\right|_{\mathrm{qsym}}^p}.
	\end{align*}
	For $0 < p \le 2$, the inequalities are reversed.
\end{thm}
Again, choosing the \(2\times2\) normalized  Hadamard isometry gives the familiar two-variable form.
	\begin{cor}
	Let $A,B\in \Mn$. Then for $p\ge 2$, 
	\[
\unorm{\qsym{A+B}^p+\qsym{A-B}^p}  \le 2^{p-1}\unorm{ \qsym{A}^p+\qsym{B}^p}.
	\]
	For $0<p\le 2$, the inequality is reversed.
\end{cor}
Taking the trace in Theorem~\ref{thm:uin} or the trace norm in Theorem~\ref{thm:uin} recovers the following Schatten \(p\)-norm version.
\begin{cor}
	Let $B_i, A_j \in  \mathbb M_N$ for $1\le i\le m, 1\le j\le n$ satisfy $(B_1,\ldots,B_m)^T=U(A_1,\ldots,A_n)^T$ for some isometry $U=(u_{ij})$. Then for $p\ge 2$, 
	\[
	m^{1 - \frac{p}{2}}\sum_{j=1}^n \||A_j|_{\mathrm{qsym}}\|_p^p \le \sum_{i=1}^m \||B_i|_{\mathrm{qsym}}\|_p^p \le n^{\frac{p}{2} - 1} \sum_{j=1}^n \||A_j|_{\mathrm{qsym}}\|_p^p.
	\]
	For $0 < p \le 2$, the inequalities are reversed.
\end{cor}
	\begin{cor}
	Let $A,B\in \Mn$. Then for $p\ge 2$, 
	\[
\| \qsym{A+B}\|_p^p+\|\qsym{A-B}\|_p^p\le 2^{p-1}(\|\qsym{A}\|_p^p+\|\qsym{B}\|_p^p) .
	\]
	For $0<p\le 2$, the inequality is reversed.
\end{cor}
\subsection{Questions of Bourin and Lee on symmetric moduli and Euler-type identities}
	For $A,B\in\Mn$ and $x\in\mathbb R$, we write
$A\nabla_x B := (1-x)A + xB$
for the weighted arithmetic mean of $A$ and $B$.
We begin by recalling a unitary-orbit refinement of the weighted parallelogram identity due to Bourin and Lee \cite{BL26}. 
\begin{thm}[{\cite[Corollary~2.2]{BL26}}]\label{thm:BL-C22}
	Let $A,B\in \mathbb M_n$ and $0\le x\le 1$. Then there exist isometries $U,V,S,T\in \mathbb M_{2n,n}$ such that
	\[
	|A\oplus B|^{2}
	=
	U\,|A\nabla_x B|^{2}U^{*}
	+
	V\,|B\nabla_x A|^{2}V^{*}
	+
	x(1-x)\Bigl\{\,S\,|A-B|^{2}S^{*}+T\,|A-B|^{2}T^{*}\Bigr\}.
	\]
\end{thm}

Taking traces in Theorem~\ref{thm:BL-C22} yields 
\[
\|A\oplus B\|_2^2=\|A\nabla_x B\|_2^2+\|B\nabla_x A\|_2^2+2x(1-x)\|A-B\|_2^2
\] for $0\le x\le 1$.
Since the right-hand side of that trace identity is a polynomial in $x$, it extends to all
$x\in\mathbb R$. This motivates the following question.

\begin{ques}[{\cite[Question~2.3]{BL26}}]\label{ques:BL-Q23}
	Does Theorem~\ref{thm:BL-C22} hold for every $x\in\mathbb R$?
\end{ques}

We first answer Question~\ref{ques:BL-Q23} in the negative: the restriction $0\le x\le 1$ in
Theorem~\ref{thm:BL-C22} cannot be removed and therefore is sharp (see Theorem~\ref{thm:Q23-all-x-outside} below).

\begin{thm}\label{thm:Q23-all-x-outside}
	Fix $x\in\mathbb R\setminus[0,1]$. Then Theorem~\ref{thm:BL-C22} fails for this $x$
	and some $A,B\in\mathbb{M}_n$. In fact, a counterexample exists already for $n=1$.
\end{thm}

Bourin and Lee \cite[Corollary~4.3]{BL26} obtained the following unitary-orbit estimate for the quadratic symmetric modulus, which can be regarded as  a matrix analogue of the scalar inequality $|z| \le |a| + |b|$ for a complex number $z = a + ib$. 
\begin{thm}[Bourin--Lee]\label{thm:Bourin_Lee}
	Let $Z\in\Mn$. Then there exist unitary matrices $U,V\in\Mn$ such that
	\[
	\sqrt{\frac{|Z|^{2}+|Z^{*}|^{2}}{2}}
	\le
	U|\Ree\, Z|U^{*}+V|\Ima\, Z|V^{*}.
	\]
\end{thm}

Clearly, since the function $t \mapsto t^{p/2}$ is operator concave on $[0,\infty)$ for $0<p<2$, we have
\[
\frac{|Z|^{p}+|Z^{*}|^{p}}{2}
=
\frac{(|Z|^2)^{p/2}+(|Z^{*}|^2)^{p/2}}{2}
\le
\left( \frac{|Z|^2+|Z^{*}|^2}{2}\right)^{p/2}.
\]
Since both sides are positive semidefinite and the function $t\mapsto t^{1/p}$ is increasing on $[0,\infty)$, Weyl's monotonicity theorem yields
\[
\lambda_j\!\left(\left(\frac{|Z|^{p}+|Z^{*}|^{p}}{2}\right)^{1/p}\right)
\le
\lambda_j\!\left(\sqrt{\frac{|Z|^2+|Z^{*}|^2}{2}}\right)
\]
for all $j$, where $\lambda_j(\cdot)$ denotes the $j$-th largest eigenvalue. Therefore,
\[
\left(\frac{|Z|^{p}+|Z^{*}|^{p}}{2}\right)^{1/p}
\le
W\sqrt{\frac{|Z|^{2}+|Z^{*}|^{2}}{2}}\,W^*
\]
for some unitary matrix $W$.
Therefore, by Theorem~\ref{thm:Bourin_Lee}, for any $0<p<2$, we obtain
\[
\left(\frac{|Z|^{p}+|Z^{*}|^{p}}{2}\right)^{1/p}
\le
U|\Ree Z|U^{*}+V|\Ima Z|V^{*}.
\]
Bourin and Lee \cite{BL26} then raised the following question concerning whether the exponent $2$ is optimal.

\begin{ques}[{\cite[Question~4.6]{BL26}}]\label{ques:1}
	Let $p>0$. Suppose that, for every $Z\in\Mn$, one can find unitary matrices $U,V\in\Mn$ such that
	\begin{equation}\label{eq:Bourin_Lee_p_le_2}
		\left(\frac{|Z|^{p}+|Z^{*}|^{p}}{2}\right)^{1/p}
		\le
		U\,|\Ree\,Z|\,U^{*}+V\,|\Ima\,Z|\,V^{*}.
	\end{equation}
	Must we necessarily have $p\le 2$?
\end{ques}

Our next result answers Question~\ref{ques:1} affirmatively in every noncommutative dimension $n\ge2$ by constructing a $2\times2$ counterexample for $p>2$.

\begin{thm}\label{thm:main1}
	Let $n\ge2$ and $p>0$.
	If for every $Z\in\Mn$ there exist unitary matrices $U,V\in\Mn$ such that \eqref{eq:Bourin_Lee_p_le_2} holds, then necessarily $p\le2$.
\end{thm}

\begin{rem}\label{rem:n=1}
	For $n=1$, the statement of Question~\ref{ques:1} is different:
	for a complex number $z$, the left-hand side of \eqref{eq:Bourin_Lee_p_le_2} equals $|z|$ for every $p>0$, while the right-hand side equals $|\Re z|+|\Im z|\ge |z|$.
	Hence \eqref{eq:Bourin_Lee_p_le_2} holds for all $p>0$ in dimension $1$.
	The necessity $p\le2$ is therefore a genuinely noncommutative phenomenon requiring $n\ge2$.
\end{rem}
Bourin and Lee \cite{BL26} further raised the following question on singular values.

\begin{rem}\label{rem:BL26-index-slip}
	In \cite[Corollary~4.7]{BL26}, the indices on the right-hand side appear with a typographical slip.
	Starting from the operator inequality in \cite[Corollary~4.3]{BL26} and applying the Weyl's inequality for singular values,
	\[
	\mu_{1+j+k}^{\downarrow}(A+B)\le \mu_{1+j}^{\downarrow}(A)+\mu_{1+k}^{\downarrow}(B)\qquad (A,B\ge0),
	\]
 where $\mu_j^{\downarrow}(\cdot)$ denotes the $j$-th largest singular value,
	one obtains the corrected bound
	\begin{equation*}
		\mu_{1+j+k}^{\downarrow}\!\left(\sqrt{\frac{|Z|^{2}+|Z^{*}|^{2}}{2}}\right)
		\le
		\mu_{1+j}^{\downarrow}(\Ree Z)+\mu_{1+k}^{\downarrow}(\Ima Z).
	\end{equation*}
	Accordingly, the natural corrected version of \cite[Question~4.8]{BL26} asks whether there exists $Z\in\mathbb K$ such that for every $p>2$ one has
	\[
	\mu_{1+j+k}^{\downarrow}\!\left(\left(\frac{|Z|^{p}+|Z^{*}|^{p}}{2}\right)^{1/p}\right)
	>
	\mu_{1+j}^{\downarrow}(\Ree Z)+\mu_{1+k}^{\downarrow}(\Ima Z)
	\]
	for some integers $j,k\ge0$.
\end{rem}

\begin{ques}[{\cite[Question~4.8]{BL26}}]\label{ques:2}
	Does there exist $Z\in\mathbb{K}$ such that for every $p>2$ one has
	\begin{equation}\label{eq:integers_j_k}
		\mu_{1+j+k}^{\downarrow}\!\left(\left(\frac{|Z|^{p}+|Z^{*}|^{p}}{2}\right)^{1/p}\right)
		>
		\mu_{1+j}^{\downarrow}(\Ree Z)+\mu_{1+k}^{\downarrow}(\Ima Z)
	\end{equation}
	for some pair of integers $j,k\ge0$?
\end{ques}

We answer Question~\ref{ques:2} in the affirmative.

\begin{thm}\label{thm:main2}
	There exists $Z\in\mathbb K$ and fixed integers $j,k\ge 0$ such that for every $p>2$,
	\eqref{eq:integers_j_k} holds.
	In our construction one may take $(j,k)=(2,0)$.
\end{thm}
Bourin and Lee \cite{BL26}  finally asked for matrix analogues of Euler's quadrilateral identity.

\begin{ques}[{\cite[Question~4.10]{BL26}}]\label{ques:3}
	What are the matrix versions, if any, of Euler's quadrilateral identity
	\[
	\|x+y+z\|_2^{2}+\|x\|_2^{2}+\|y\|_2^{2}+\|z\|_2^{2}
	=
	\|x+y\|_2^{2}+\|y+z\|_2^{2}+\|z+x\|_2^{2}
	\]
	for three points $x,y,z\in\mathbb{C}^{n}$?
\end{ques}

A direct computation yields the following Euler-type operator identity.

\begin{prop}[Euler operator identity]\label{prop:Euler-operator}
	Let $A,B,C\in\Mn$. Then
	\begin{equation}\label{eq:Euler-op}
		|A+B+C|^{2}+|A|^{2}+|B|^{2}+|C|^{2}
		=
		|A+B|^{2}+|B+C|^{2}+|C+A|^{2}.
	\end{equation}
\end{prop}
Bourin and Lee \cite{BL26} suggested exploring isometry/unitary-orbit refinements of Euler-type identities.
Our following three results give isometry-orbit refinements of \eqref{eq:Euler-op}.

\begin{thm}\label{thm:main3}
	Let $A,B,C\in\Mn$. Then there exist isometries $U_{ij}\in \mathbb{M}_{4n,n}$, $1\le i,j\le 4$, such that
	\begin{align}\label{eq:main3}
		|A+B|^2\oplus|B+C|^2\oplus|C+A|^2\oplus 0
		&=
		\frac14\sum_{j=1}^4U_{1j}\,|A+B+C|^2\,U_{1j}^*
		+\frac14\sum_{j=1}^4U_{2j}\,|A|^2\,U_{2j}^*\nonumber\\
		&\quad+
		\frac14\sum_{j=1}^4U_{3j}\,|B|^2\,U_{3j}^*
		+
		\frac14\sum_{j=1}^4U_{4j}\,|C|^2\,U_{4j}^*.
	\end{align}
\end{thm}

\begin{thm}\label{thm:first}
	Let $A,B,C\in \Mn$. Then there exist three isometries $U_1,U_2,U_3\in\mathbb{M}_{3n,n}$ such that
	\[
	|A+B|^2\oplus|B+C|^2\oplus|A+C|^2
	= \frac{1}{3}\sum_{k=1}^3U_k\,\left( |A+B+C|^2+|A|^2+|B|^2+|C|^2\right) \,U_k^*.
	\]
\end{thm}

\begin{thm}\label{thm:second}
	Let $A,B,C\in \Mn$. Then there exist four isometries $U_1,U_2,U_3,U_4\in\mathbb{M}_{4n,n}$ such that
	\[
	|A+B+C|^2\oplus|A|^2\oplus|B|^2\oplus|C|^2
	=
	\frac{1}{4}\sum_{k=1}^4 U_k\,\left( |A+B|^2+|B+C|^2+|A+C|^2\right) \,U_k^*.
	\]
\end{thm}
\begin{rem}
	The Euler-type orbit identities in Proposition~\ref{prop:Euler-operator} and Theorems~\ref{thm:first} and~\ref{thm:second}
	also admit quadratic symmetric modulus versions.
	Indeed, these two results extend verbatim to rectangular matrices with a fixed number of columns,
	since their proofs only involve the positive matrices
	\(
	|A|^2=A^*A
	\)
	and the Euler identity \eqref{eq:Euler-op},
	which remains valid in the rectangular setting.
	Applying these rectangular versions to the lifting
$
	\Gamma(T):=\frac1{\sqrt2}\binom{T}{T^*}\in\mathbb M_{2n,n},
$
	we obtain the corresponding analogues for quadratic symmetric modulus.
\end{rem}

Naturally, we ask the following question.
\begin{ques}\label{ques:3n-four-isometries}
	Does there exist isometries $U_1,U_2,U_3, U_4\in \mathbb{M}_{3n,n}$  such that
	\begin{align*}
		|A+B|^2\oplus|B+C|^2\oplus|A+C|^2
		&=U_1\,|A+B+C|^2\,U_{1}^*
		+U_{2}\,|A|^2\,U_{2}^*\\
		&\quad+U_{3}\,|B|^2\,U_{3}^*
		+U_{4}\,|C|^2\,U_{4}^* \ ?
	\end{align*}
\end{ques}
Unfortunately, the answer to Question~\ref{ques:3n-four-isometries} is negative.
In fact,	take $A=B=C=I_n$. Then
\[
|A+B|^2=|2I_n|^2=4I_n,\qquad |A+B+C|^2=|3I_n|^2=9I_n,
\qquad |A|^2=|B|^2=|C|^2=I_n.
\]
Hence the desired identity would become
\[
4I_{3n}
=
9U_1U_1^*+U_2U_2^*+U_3U_3^*+U_4U_4^*.
\]
Since all terms on the right-hand side are positive semidefinite, we have
\[
4I_{3n}\ge 9U_1U_1^*.
\]
But $U_1U_1^*$ is an orthogonal projection (of rank $n$), so $\|9U_1U_1^*\|_\infty=9$ while
$\|4I_{3n}\|_\infty=4$, a contradiction. Therefore such isometries cannot exist in general.

\subsection{Future directions}
		After early versions of this paper were completed, some related works appeared. Bourin and Lee \cite{BL26+} established triangle inequalities for unitarily invariant norms on arithmetic symmetric modulus, while the author \cite{Zha26} studied triangle inequalities for unitarily invariant norms and Schatten \(p\)-norms on two symmetric moduli.  Future directions for further study include the Lipschitz continuity of these moduli and related problems.
		
	\subsection{Organization of this paper} The paper is organized as follows.
	Section~\ref{sec:counterexample-bhl-sym} answers Question~\ref{ques:Bourin-Harada-Lee} and proves Theorem~\ref{thm:non-exist-sym}.
	Section~\ref{sec:rectangular-thompson} establishes a rectangular Thompson's inequality and uses it to prove Theorems~\ref{thm:qsym-thompson-matrix} and~\ref{thm:equality-case-matrix}.
	Section~\ref{sec:infinite-dimensional} extends the quadratic symmetric Thompson inequality to operators on Hilbert spaces and proves the corresponding equality characterization.
	Section~\ref{sec:subadditivity} proves the concave-function subadditivity theorem for the quadratic symmetric modulus and derives its norm and trace consequences.
	Section~\ref{sec:condition-number} obtains condition-number estimates for the symmetric moduli and shows that the coefficient is sharp.
	Section~\ref{sec:clarkson} develops Clarkson--McCarthy type inequalities for the quadratic symmetric modulus, both in operator order and for unitarily invariant norms.
	Sections~\ref{sec:proof-q23}, \ref{sec:proof-main1}, and~\ref{sec:proof-main2} answer three questions raised in~\cite{BL26}.
	Finally, Section~\ref{sec:euler-isometry} establishes Euler-type operator identities together with several isometry-orbit refinements.

\subsection*{Acknowledgments} The author is very grateful to Professor Bourin for his valuable suggestions.
This work is supported by the China Scholarship Council, the Young Elite Scientists Sponsorship Program for PhD Students (China Association for Science and Technology), and the Fundamental Research Funds for the Central Universities at Xi'an Jiaotong University (Grant No.~xzy022024045). 
\section{A counterexample to Question~\ref{ques:Bourin-Harada-Lee} and proof of Theorem~\ref{thm:non-exist-sym}}
\label{sec:counterexample-bhl-sym}
	
In this section, we answer Question~\ref{ques:Bourin-Harada-Lee} and establish Theorem~\ref{thm:non-exist-sym}.
	
\noindent\emph{Answer to Question~\ref{ques:Bourin-Harada-Lee}.}
The answer is negative. Indeed, consider
\[
f(t)=t-t^2,\qquad t\ge0.
\]
Then \(f\) is concave on \([0,\infty)\) and \(f(0)=0\). Let
\[
A=\begin{pmatrix}1&0\\0&0\end{pmatrix},
\qquad
B=\begin{pmatrix}c^2&cs\\ cs&s^2\end{pmatrix},
\qquad 0<c<1,\quad s=\sqrt{1-c^2}.
\]
Since \(A\) and \(B\) are rank-one projections, we have \(A^2=A\) and \(B^2=B\), and therefore
\[
f(A)=A-A^2=0,
\qquad
f(B)=B-B^2=0.
\]
Thus, for all unitaries \(U,V\),
\[
Uf(A)U^*+Vf(B)V^*=0.
\]
If the conclusion of Theorem~\ref{thm:Bourin--Uchiyama} were valid under the weaker assumption
\(f(0)=0\) and concavity alone, then we would have \(f(A+B)\le0\).

However, \(A+B\) has trace \(2\) and determinant \(1-c^2\), so its eigenvalues are \(1\pm c\).
Consequently, the eigenvalues of \(f(A+B)\) are
\[
f(1+c)=-c(1+c)
\qquad\text{and}\qquad
f(1-c)=c(1-c).
\]
Since \(0<c<1\), the latter is strictly positive. Hence \(f(A+B)\) has a positive eigenvalue, so
\(f(A+B)\nleq0\). Therefore there do not exist unitaries \(U,V\) such that
\[
f(A+B)\le Uf(A)U^*+Vf(B)V^*.
\]
This proves that Theorem~\ref{thm:Bourin--Uchiyama} does not extend to arbitrary concave functions
satisfying only \(f(0)=0\). In particular, the monotonicity assumption cannot be removed.\qed
	\begin{proof}[Proof of Theorem~\ref{thm:non-exist-sym}]
		Take
		\[
		A=\begin{pmatrix}-1&-1\\0&-1\end{pmatrix},
		\qquad
		B=\begin{pmatrix}0&-1\\0&0\end{pmatrix}\in \mathbb{M}_2.
		\]
		A direct computation gives
		\[
		|A|_{\mathrm{sym}}=\frac1{2\sqrt5}\begin{pmatrix}5&2\\2&5\end{pmatrix},
		\qquad
		|B|_{\mathrm{sym}}=\frac12 I_2,
		\qquad
		|A+B|_{\mathrm{sym}}=\frac1{\sqrt2}\begin{pmatrix}2&1\\1&2\end{pmatrix}.
		\]
		Hence
		\[
		\bigl\||A+B|_{\mathrm{sym}}\bigr\|_{\infty}=\frac{3}{\sqrt2},
		\qquad
		\bigl\||A|_{\mathrm{sym}}\bigr\|_{\infty}=\frac{7}{2\sqrt5},
		\qquad
		\bigl\||B|_{\mathrm{sym}}\bigr\|_{\infty}=\frac12,
		\]
		so
		\[
		\bigl\||A+B|_{\mathrm{sym}}\bigr\|_{\infty}
		>
		\bigl\||A|_{\mathrm{sym}}\bigr\|_{\infty}+\bigl\||B|_{\mathrm{sym}}\bigr\|_{\infty}
		\qquad
		\left(\text{i.e. }\ \frac{3}{\sqrt2}>\frac{7}{2\sqrt5}+\frac12\right).
		\]
		If the inequality in Theorem~\ref{thm:non-exist-sym} held for some unitaries $U,V$, then taking operator
		norms would yield
		\[
		\bigl\||A+B|_{\mathrm{sym}}\bigr\|_{\infty}
		\le
		\bigl\|U|A|_{\mathrm{sym}}U^*+V|B|_{\mathrm{sym}}V^*\bigr\|_{\infty}
		\le
		\bigl\||A|_{\mathrm{sym}}\bigr\|_{\infty}+\bigl\||B|_{\mathrm{sym}}\bigr\|_{\infty},
		\]
		a contradiction. 
	\end{proof}
\section{A rectangular Thompson's inequality and proof of Theorem~\ref{thm:qsym-thompson-matrix}}
\label{sec:rectangular-thompson}

In this section, we establish a rectangular version of Thompson's inequality, which will be the main tool in the proof of Theorem~\ref{thm:qsym-thompson-matrix}.
The key idea is to reduce the rectangular setting to the square one by means of the polar decomposition: for a rectangular matrix $X\in\mathbb M_{m,n}$,
one can write $X=W|X|$ with a partial isometry $W$ whose initial projection coincides with the support projection of $|X|$.
We then use the rectangular Thompson's inequality to prove Theorem~\ref{thm:qsym-thompson-matrix}.

\begin{lem}\label{lem:svd-support}
	Let $A\in\mathbb M_{m,n}$. Then there exists a partial isometry $W\in\mathbb M_{m,n}$
	such that
	\[
A=W|A|,\qquad |A|=(A^*A)^{1/2},
	\]
	and
	\[
	W^*W=\mathrm{supp}(|A|),\qquad  	W^*W\le I_n, WW^*\le I_m,\qquad (W^*W)\,|A|=|A|,
	\]
	where $\mathrm{supp}(\cdot)$  means  the support projection.
\end{lem}

\begin{proof}
	Take a singular value decomposition $A=U\Sigma V^*$, where $U\in\mathbb M_{m}$ and $V\in\mathbb M_{n}$ are unitary and
	\[
	\Sigma=
	\begin{pmatrix}
		D & 0\\
		0 & 0
	\end{pmatrix}\in\mathbb M_{m,n},
	\qquad
	D=\diag(s_1,\dots,s_r),\ \ s_1\ge\cdots\ge s_r>0,\ \ r=\rank(A).
	\]
	Then
	\[
A^*A = V(\Sigma^*\Sigma)V^*,\qquad 
	\Sigma^*\Sigma=
	\begin{pmatrix}
		D^2 & 0\\
		0 & 0
	\end{pmatrix},
	\]
	hence
	\[
	|A|=(A^*A)^{1/2}
	=
	V
	\begin{pmatrix}
		D & 0\\
		0 & 0
	\end{pmatrix}
	V^*.
	\]
	Define
	$
	W:=U
	\begin{pmatrix}
		I_r & 0\\
		0 & 0
	\end{pmatrix}
	V^* \in\mathbb M_{m,n}.
	$
	Clearly, $W^*W\le I_n, WW^*\le I_m$ and
	\[
	W|A|
	=
	U 	\begin{pmatrix}
		I_r & 0\\
		0 & 0
	\end{pmatrix}V^* \cdot V
	\begin{pmatrix}
		D & 0\\
		0 & 0
	\end{pmatrix}
	V^*
	=
	U
	\begin{pmatrix}
		D & 0\\
		0 & 0
	\end{pmatrix}
	V^*
	=
	U\Sigma V^*
	=
	A.
	\]
	Moreover,
\[
(W^*W)\,|A|
=
\Bigl(V\begin{pmatrix}
	I_r & 0\\
	0 & 0
\end{pmatrix}V^*\Bigr)
\Bigl(V\begin{pmatrix}
	D & 0\\
	0 & 0
\end{pmatrix}V^*\Bigr)
=
V\begin{pmatrix}
	D & 0\\
	0 & 0
\end{pmatrix}V^*
=
|A|.
\]
	That is, $W^*W$ is exactly the support projection of $|A|$.
\end{proof}

Lemma~\ref{lem:svd-support} allows us to choose the partial isometry in the polar decomposition so that it acts as the identity on the range of $|A|$.
This normalization is exactly what we need to turn $|A+B|$ into the absolute value of a \emph{square} matrix, to which the classical Thompson inequality applies.

	\begin{prop}\label{prop:Thompson-rect}
	Let $A,B\in\mathbb M_{m,n}$. Then there exist two unitaries $U,V$ such that
	\[
	|A+B|\le U|A|U^*+V|B|V^*.
	\]
\end{prop}
\begin{proof}
	By using Lemma~\ref{lem:svd-support},	 there exists a partial isometry $W$ such that $A+B=W\,|A+B|$,
	$W^*W$ is the support projection of $|A+B|$ and $W^*W\,|A+B|=|A+B|$. Therefore
	\[
	W^*(A+B)=W^*\left( W|A+B|\right) =|A+B|.
	\]
	In particular,
	\begin{equation}\label{eq:polar-decomposition-equality}
		|A+B|=|W^*(A+B)|=|W^*A+W^*B|.
	\end{equation}
	Now $W^*A$ and $W^*B$ are $n\times n$ (square) matrices. Applying Thompson's inequality (Theorem~\ref{thm:thompson}) to the square matrices
	$W^*A$ and $W^*B$, there exist unitaries $U,V\in\mathbb M_n$ such that
	\begin{equation}\label{eq:Thompson-square-applied}
		|W^*A+W^*B|
		\le
		U\,|W^*A|\,U^*+V\,|W^*B|\,V^*.
	\end{equation}
	Since the left-hand side equals $|A+B|$, it remains to compare $|W^*A|$ with $|A|$.
	
	Since $WW^*\le I_m$, we have
	\[
	|W^*A|^2=(W^*A)^*(W^*A)=A^*WW^*A\le A^*A=|A|^2.
	\]
	By the operator monotonicity of $t\mapsto t^{1/2}$ on $[0,\infty)$, this implies $|W^*A|\le |A|$.
	Similarly, $|W^*B|\le |B|$. Conjugating preserves the order, hence
	\[
	U|W^*A|U^*\le U|A|U^*,\qquad V|W^*B|V^*\le V|B|V^*.
	\]
	Substituting into \eqref{eq:Thompson-square-applied} and combining with \eqref{eq:polar-decomposition-equality} yields our desired result.
\end{proof}

We now prove Theorem~\ref{thm:qsym-thompson-matrix},
it follows from an application of Proposition~\ref{prop:Thompson-rect}
to a suitable linear embedding that converts the quadratic symmetrization into a rectangular modulus.

\begin{proof}[Proof of Theorem~\ref{thm:qsym-thompson-matrix}]
 Define the linear map $\Gamma:\mathbb M_n\to \mathbb M_{2n,n}$ by
	\[
	\Gamma(T):=\frac{1}{\sqrt2}\begin{pmatrix}T\\ T^*\end{pmatrix}.
	\]
	Then
	\[
\Gamma(A+B)=\Gamma(A)+\Gamma(B)\quad \text{ and }\quad	|\Gamma(T)|=\,|T|_{\mathrm{qsym}}.
	\]
 Applying rectangular Thompson's inequality (Proposition~\ref{prop:Thompson-rect})
	yields unitaries $U,V\in\Mn$ such that
\begin{align*}
		|\Gamma(A+B)|&=|\Gamma(A)+\Gamma(B)|\\
		&\le U|\Gamma(A)|U^*+V|\Gamma(B)|V^*.
\end{align*}
That is,
	\[
	|A+B|_{\mathrm{qsym}}\le U|A|_{\mathrm{qsym}}U^*+V|B|_{\mathrm{qsym}}V^*.
	\]
\end{proof}
We now turn to the equality case.
The argument is most transparent after embedding the rectangular problem into a square one, where Thompson's equality theorem is available.
\begin{proof}[Proof of Theorem~\ref{thm:equality-case-matrix}]
	Define
	\[
	A_0:=\Gamma(A)\in \mathbb{M}_{2n,n},\qquad B_0:=\Gamma(B)\in \mathbb{M}_{2n,n}.
	\]
Then
	\[
	|A_0|=\qsym{A},\qquad |B_0|=\qsym{B},\qquad |A_0+B_0|=\qsym{A+B}.
	\]
	Hence the assumption
	\[
	\qsym{A+B}=U\qsym{A}U^*+V\qsym{B}V^*
	\]
	is equivalent to
	\begin{equation}\label{eq:rect-orbit-eq}
		|A_0+B_0|=U|A_0|U^*+V|B_0|V^* .
	\end{equation}
	
	\smallskip
	\noindent\emph{Step 1: Embed the rectangular problem into a square one.}
	Consider the block embedding of $\mathbb{M}_{2n,n}$ into $\mathbb{M}_{3n}$ given by
	\[
	\widetilde A:=\begin{pmatrix}0&0\\ A_0&0\end{pmatrix},\qquad
	\widetilde B:=\begin{pmatrix}0&0\\ B_0&0\end{pmatrix}
	\quad\text{in } \mathbb{M}_{3n},
	\]
	with respect to the decomposition $\mathbb{C}^{3n}\cong \mathbb{C}^{n}\oplus\mathbb{C}^{2n}$.
	A direct computation shows
	\[
	|\widetilde A|=\diag(|A_0|,0),\qquad
	|\widetilde B|=\diag(|B_0|,0),\qquad
	|\widetilde A+\widetilde B|=\diag(|A_0+B_0|,0).
	\]
	Let
	\[
	\widetilde U:=\diag(U,I_{2n}),\qquad \widetilde V:=\diag(V,I_{2n})
	\quad\text{in }\mathbb{M}_{3n}.
	\]
	Then \eqref{eq:rect-orbit-eq} implies
	\[
	|\widetilde A+\widetilde B|
	=\diag(|A_0+B_0|,0)
	=\diag(U|A_0|U^*+V|B_0|V^*,0)
	=\widetilde U|\widetilde A|\widetilde U^*+\widetilde V|\widetilde B|\widetilde V^* .
	\]
	
	\smallskip
	\noindent\emph{Step 2: Apply Thompson's equality theorem in $\mathbb{M}_{3n}$.}
	By Theorem~\ref{thm:thompson-2}, there exists a unitary $\widetilde W\in \mathbb{M}_{3n}$ such that
	\[
	\widetilde A=\widetilde W|\widetilde A|,\qquad \widetilde B=\widetilde W|\widetilde B|.
	\]
	Write $\widetilde W$ in block form (with respect to $\mathbb C^{n}\oplus\mathbb C^{2n}$):
	\[
	\widetilde W=\begin{pmatrix}W_{11}&W_{12}\\ W_{21}&W_{22}\end{pmatrix}.
	\]
	Since $|\widetilde A|=\diag(|A_0|,0)$, the identity $\widetilde A=\widetilde W|\widetilde A|$ gives
	\[
	\begin{pmatrix}0&0\\ A_0&0\end{pmatrix}
	=
	\begin{pmatrix}W_{11}&W_{12}\\ W_{21}&W_{22}\end{pmatrix}
	\begin{pmatrix}|A_0|&0\\ 0&0\end{pmatrix}
	=
	\begin{pmatrix}W_{11}|A_0|&0\\ W_{21}|A_0|&0\end{pmatrix}.
	\]
	Hence
	\[
	A_0=W_{21}|A_0|
	\qquad\text{and}\qquad
	W_{11}|A_0|=0.
	\]
	Similarly, from $\widetilde B=\widetilde W|\widetilde B|$ we obtain
	\[
	B_0=W_{21}|B_0|
	\qquad\text{and}\qquad
	W_{11}|B_0|=0.
	\]
	
	\smallskip
	\noindent\emph{Step 3: Construct the common partial isometry.}
	Let
	\[
	P:=\supp(|A_0|+|B_0|)\in\mathbb{M}_n.
	\]
	Then $P|A_0|=|A_0|$ and $P|B_0|=|B_0|$, so if we define
	\[
	W:=W_{21}P\in\mathbb{M}_{2n,n},
	\]
	then
	\[
	W|A_0|=W_{21}P|A_0|=W_{21}|A_0|=A_0,
	\qquad
	W|B_0|=W_{21}P|B_0|=W_{21}|B_0|=B_0.
	\]
	Moreover, since $W_{11}|A_0|=0$ and $W_{11}|B_0|=0$, we have
	\[
	W_{11}(|A_0|+|B_0|)=0,
	\]
	hence $W_{11}P=0$. Therefore
	\[
	\widetilde W\,\diag(P,0)
	=
	\begin{pmatrix}W_{11}P&0\\ W_{21}P&0\end{pmatrix}
	=
	\begin{pmatrix}0&0\\ W&0\end{pmatrix}.
	\]
	Because $\widetilde W$ is unitary, $\widetilde W\,\diag(P,0)$ is a partial isometry with initial projection $\diag(P,0)$.
	Taking adjoints and multiplying yields
	\[
	\begin{pmatrix}0& W^*\\ 0&0\end{pmatrix}
	\begin{pmatrix}0&0\\ W&0\end{pmatrix}
	=
	\diag(P,0),
	\]
	and hence
	\[
	W^*W=P.
	\]
	So $W$ is a partial isometry in $\mathbb{M}_{2n,n}$.
	
	Finally, since $A_0=\Gamma(A)$, $B_0=\Gamma(B)$, $|A_0|=\qsym{A}$, and $|B_0|=\qsym{B}$, we conclude that
	\[
	\Gamma(A)=W\,\qsym{A},
	\qquad
	\Gamma(B)=W\,\qsym{B}.
	\]
	Equivalently, the polar decompositions of $\Gamma(A)$ and $\Gamma(B)$ share the same partial isometry factor $W$.
\end{proof}
\section{Infinite-dimensional extensions and equality cases}
\label{sec:infinite-dimensional}
In this section, we discuss the infinite-dimensional case of the triangle inequality for the quadratic symmetric modulus.

We begin by recording a rectangular version of Theorem~\ref{thm:AAP} for operators in $\B(\Hh,\Kk)$.
	\begin{prop}\label{prop:rect-AAP}
		Let $\Hh,\Kk$ be complex Hilbert spaces and let $A,B\in\B(\Hh,\Kk)$.
		Then there exist isometries $U,V\in\B(\Hh)$ such that
		\[
		\absop{A+B}\ \le\ U\,\absop{A}\,U^* \;+\; V\,\absop{B}\,V^* .
		\]
	\end{prop}
	
	\begin{proof}
		Let $A+B=W\,\absop{A+B}$ be the polar decomposition, where $W\in\B(\Hh,\Kk)$ is a partial isometry.
	Recall that for an operator $S$ we write $\ran(S)$ for its range.
	For a positive operator $S$ we denote by $\supp(S)$ its support projection,
	i.e.\ the orthogonal projection onto $\overline{\ran(S)}$.
		Then
		\[
		W^*(A+B)=W^*W\,\absop{A+B}=\absop{A+B},
		\]
		since $W^*W=\supp(\absop{A+B})$.
		Hence
		\[
		\absop{A+B}=\absop{W^*(A+B)}=\absop{W^*A+W^*B}.
		\]
		Now $W^*A,W^*B\in\B(\Hh)$ are operators on $\Hh$, so by  Akemann--Anderson--Pedersen's inequality (Theorem~\ref{thm:AAP})
		there exist isometries $U,V\in\B(\Hh)$ such that
		\[
		\absop{W^*A+W^*B}\ \le\ U\,\absop{W^*A}\,U^* \;+\; V\,\absop{W^*B}\,V^* .
		\]
		Finally, since $WW^*\le \Id_{\Kk}$ we have
		\[
		\absop{W^*A}^2=(W^*A)^*(W^*A)=A^*WW^*A\le A^*A=\absop{A}^2,
		\]
		hence $\absop{W^*A}\le\absop{A}$ by operator monotonicity of $t\mapsto t^{1/2}$.
		Similarly $\absop{W^*B}\le\absop{B}$.
		Conjugation preserves order, so
		$U\absop{W^*A}U^*\le U\absop{A}U^*$ and $V\absop{W^*B}V^*\le V\absop{B}V^*$.
		Combining the last three displays yields the desired inequality.
	\end{proof}
	Now, we give a proof of Theorem~\ref{thm:qsym-thompson-operator}.
\begin{proof}[Proof of Theorem~\ref{thm:qsym-thompson-operator}]
	Define the linear lifting
	\[
	\Gamma:\B(\Hh)\to\B(\Hh,\Hh\oplus\Hh),\qquad
	\Gamma(T):=\frac{1}{\sqrt2}\binom{T}{T^*}.
	\]
		Then
	\[
	\Gamma(A+B)=\Gamma(A)+\Gamma(B)\quad \text{ and }\quad	|\Gamma(T)|=\,|T|_{\mathrm{qsym}}.
	\]
Applying Proposition~\ref{prop:rect-AAP}
yields isometries $U,V\in\B(\Hh)$ such that
	\begin{align*}
		|\Gamma(A+B)|&=|\Gamma(A)+\Gamma(B)|\\
		&\le U|\Gamma(A)|U^*+V|\Gamma(B)|V^*.
	\end{align*}
	That is,
	\[
	|A+B|_{\mathrm{qsym}}\le U|A|_{\mathrm{qsym}}U^*+V|B|_{\mathrm{qsym}}V^*.
	\]
	\end{proof}
	
	\begin{rem}
		In finite dimension, every isometry is unitary. Hence Theorem~\ref{thm:qsym-thompson-operator}
		implies Theorem~\ref{thm:qsym-thompson-matrix}. 
	\end{rem}
	
Before proving Theorem~\ref{thm:main-equality}, we need the standard two-space extension of Theorem~\ref{thm:AH-classical} to operators between Hilbert spaces.
	
	\begin{prop}\label{prop:AH-two-space}
		Let $\Hh,\Kk$ be complex Hilbert spaces and let $A,B\in\B(\Hh,\Kk)$.
		If
		\[
		\absop{A+B}=\absop{A}+\absop{B},
		\]
		then there exists a partial isometry $W\in\B(\Hh,\Kk)$ such that
		\[
		A=W\absop{A},\qquad B=W\absop{B}.
		\]
	\end{prop}
	
	\begin{proof}
		Consider $\Hh\oplus\Kk$ and define block operators
		\[
		\widetilde A:=\begin{pmatrix}0&0\\ A&0\end{pmatrix},
		\qquad
		\widetilde B:=\begin{pmatrix}0&0\\ B&0\end{pmatrix}
		\quad\text{in }\B(\Hh\oplus\Kk).
		\]
		Then $\absop{\widetilde A}=\diag(\absop{A},0)$, $\absop{\widetilde B}=\diag(\absop{B},0)$, and
		$\absop{\widetilde A+\widetilde B}=\diag(\absop{A+B},0)$.
		Hence $\absop{A+B}=\absop{A}+\absop{B}$ is equivalent to
		$\absop{\widetilde A+\widetilde B}=\absop{\widetilde A}+\absop{\widetilde B}$.
		Applying Theorem~\ref{thm:AH-classical} on $\Hh\oplus\Kk$ yields a partial isometry
		$\widetilde U\in\B(\Hh\oplus\Kk)$ such that
		\[
		\widetilde A=\widetilde U\,\absop{\widetilde A},
		\qquad
		\widetilde B=\widetilde U\,\absop{\widetilde B}.
		\]
		Write $\widetilde U=\begin{pmatrix}U_{11}&U_{12}\\ U_{21}&U_{22}\end{pmatrix}$.
		From
		\[
		\widetilde A=\widetilde U\,\diag(\absop{A},0)
		\quad\text{and}\quad
		\widetilde B=\widetilde U\,\diag(\absop{B},0)
		\]
		we obtain
		\[
		U_{21}\absop{A}=A,\qquad U_{21}\absop{B}=B,
		\qquad
		U_{11}\absop{A}=0,\qquad U_{11}\absop{B}=0.
		\]
		Let $P$ be the support projection of $\absop{A}+\absop{B}$ on $\Hh$ (equivalently, the orthogonal projection onto
		$\overline{\ran(\absop{A}+\absop{B})}$). Then $P\absop{A}=\absop{A}$ and $P\absop{B}=\absop{B}$, and moreover
		$U_{11}P=0$ since $U_{11}$ vanishes on $\ran(\absop{A})$ and $\ran(\absop{B})$.
		Set $W:=U_{21}P\in\B(\Hh,\Kk)$. Then
		\[
		A=W\absop{A},\qquad B=W\absop{B}.
		\]
		Finally, let $\widetilde P:=\diag(P,0)=\supp\bigl(|\widetilde A|+|\widetilde B|\bigr)$.
		Since $\widetilde A=\widetilde U|\widetilde A|$ and $\widetilde B=\widetilde U|\widetilde B|$, we have
		$\supp(|\widetilde A|)\le \widetilde U^*\widetilde U$ and $\supp(|\widetilde B|)\le \widetilde U^*\widetilde U$,
		hence $\widetilde P\le \widetilde U^{*}\widetilde U$. Therefore $\widetilde U\widetilde P$ is a partial isometry
		with initial projection $\widetilde P$.
		Since
		\[
		\widetilde U\widetilde P=\begin{pmatrix}U_{11}P&0\\ U_{21}P&0\end{pmatrix}
		=\begin{pmatrix}0&0\\ W&0\end{pmatrix},
		\]
		it follows that $W^{*}W=P$, and therefore $W$ is a partial isometry.
	\end{proof}
	
Next, we prove Theorem~\ref{thm:main-equality}.

\begin{proof}[Proof of Theorem~\ref{thm:main-equality}]
	Clearly,
	\[
	\absop{\Gamma(A)}=\qsym{A},\qquad
	\absop{\Gamma(B)}=\qsym{B},\qquad
	\absop{\Gamma(A+B)}=\qsym{A+B},
	\]
	and
	\[
	\Gamma(A+B)=\Gamma(A)+\Gamma(B).
	\]
Thus the assumption
\[
\qsym{A+B}=\qsym{A}+\qsym{B}
\]
is equivalent to
\[
\absop{\Gamma(A)+\Gamma(B)}
=
\absop{\Gamma(A)}+\absop{\Gamma(B)}.
\]
	Applying Proposition~\ref{prop:AH-two-space} to the operators
$
	\Gamma(A),\Gamma(B)\in\B(\Hh,\Hh\oplus\Hh),
$
	we obtain a partial isometry $W\in\B(\Hh,\Hh\oplus\Hh)$ such that
	\[
	\Gamma(A)=W\,\absop{\Gamma(A)}=W\,\qsym{A},
	\qquad
	\Gamma(B)=W\,\absop{\Gamma(B)}=W\,\qsym{B}.
	\]
\end{proof}

\begin{rem}[Block form]
	Write $W=\binom{W_1}{W_2}$ with $W_1,W_2\in\B(\Hh)$.
	Then $\Gamma(A)=W\qsym{A}$ and $\Gamma(B)=W\qsym{B}$ are equivalent to
	\[
	A=\sqrt2\,W_1\,\qsym{A},\quad A^*=\sqrt2\,W_2\,\qsym{A},
	\qquad
	B=\sqrt2\,W_1\,\qsym{B},\quad B^*=\sqrt2\,W_2\,\qsym{B}.
	\]
	Thus the equality $\qsym{A+B}=\qsym{A}+\qsym{B}$ forces the \emph{same} pair $(W_1,W_2)$
	to control both $(A,A^*)$ and $(B,B^*)$ through their quadratic symmetric moduli.
\end{rem}

\begin{rem}[Why one should not expect $A=U\qsym{A}$ in general]
	Even for a single operator $A$, a factorization $A=U\qsym{A}$ with $U\in\B(\Hh)$ a partial isometry
	is generally impossible unless $A$ is very special.
	For example, if $\Hh=\mathbb{C}^n$ and $A=U\qsym{A}$, then
	\[
	A^*A=\qsym{A}\,U^*U\,\qsym{A}\le \qsym{A}^2=\frac{A^*A+AA^*}{2},
	\]
	so $A^*A\le AA^*$. Taking traces forces $A^*A=AA^*$, hence $A$ is normal.
	Therefore non-normal matrices (e.g.\ a Jordan block) cannot satisfy $A=U\qsym{A}$.
	This explains why the natural equality theorem for the  quadratic symmetric modulus is formulated for the lifted operators $\Gamma(A)$.
\end{rem}
\section{Concave-function subadditivity for the quadratic symmetric modulus}
\label{sec:subadditivity}

We begin by proving the concave-function subadditivity theorem, namely Theorem~\ref{thm:intro-main-concave}.
\begin{proof}[Proof of Theorem~\ref{thm:intro-main-concave}]
	Since $f:[0,\infty)\to[0,\infty)$ is concave, it is nondecreasing on $[0,\infty)$.
	By Theorem~\ref{thm:qsym-thompson-matrix}, there exist unitaries
	$U_0,V_0\in\mathbb M_n$ such that
	\[
	\qsym{A+B}\le U_0\,\qsym{A}\,U_0^*+V_0\,\qsym{B}\,V_0^*.
	\]
	Set
	\[
	X:=\qsym{A+B},
	\qquad
	Y:=U_0\,\qsym{A}\,U_0^*+V_0\,\qsym{B}\,V_0^*.
	\]
	Then $0\le X\le Y$, and hence
	\[
	\lambda_j(X)\le \lambda_j(Y),\qquad 1\le j\le n.
	\]
	Since $f$ is nondecreasing,
	\[
	\lambda_j\bigl(f(X)\bigr)=f\bigl(\lambda_j(X)\bigr)
	\le f\bigl(\lambda_j(Y)\bigr)
	=\lambda_j\bigl(f(Y)\bigr),
	\qquad 1\le j\le n.
	\]
	Therefore there exists a unitary $W\in\mathbb M_n$ such that
	\[
	f(X)\le W\,f(Y)\,W^*.
	\]
	
	Now apply Theorem~\ref{thm:Bourin--Uchiyama} to the positive matrices
	$U_0\,\qsym{A}\,U_0^*$ and $V_0\,\qsym{B}\,V_0^*$.
	Then there exist unitaries $W_1,W_2\in\mathbb M_n$ such that
	\begin{align*}
		f(Y)
		&\le
		W_1\,f\!\bigl(U_0\,\qsym{A}\,U_0^*\bigr)\,W_1^*
		+
		W_2\,f\!\bigl(V_0\,\qsym{B}\,V_0^*\bigr)\,W_2^* \\
		&=
		W_1U_0\,f\bigl(\qsym{A}\bigr)\,U_0^*W_1^*
		+
		W_2V_0\,f\bigl(\qsym{B}\bigr)\,V_0^*W_2^*.
	\end{align*}
	Combining the last two inequalities and setting
	\[
	U:=WW_1U_0,
	\qquad
	V:=WW_2V_0,
	\]
	we obtain
	\[
	f\bigl(\qsym{A+B}\bigr)
	\le
	Uf\bigl(\qsym{A}\bigr)U^*
	+
	Vf\bigl(\qsym{B}\bigr)V^*.
	\]
\end{proof}
Next, we prove Corollary~\ref{cor:triangle-plus}.
\begin{proof}[Proof of Corollary~\ref{cor:triangle-plus}]
	Apply Theorem~\ref{thm:intro-main-concave} to the pair $A-B$ and $B$.
	Then there exist unitaries $W,T\in\mathbb M_n$ such that
	\[
	f\bigl(\qsym{A}\bigr)
	=
	f\bigl(\qsym{(A-B)+B}\bigr)
	\le
	Wf\bigl(\qsym{A-B}\bigr)W^*
	+
	Tf\bigl(\qsym{B}\bigr)T^*.
	\]
	Conjugating by $W^*$, we obtain
	\[
	W^*f\bigl(\qsym{A}\bigr)W
	\le
	f\bigl(\qsym{A-B}\bigr)
	+
	W^*Tf\bigl(\qsym{B}\bigr)T^*W.
	\]
	Hence
	\[
	W^*f\bigl(\qsym{A}\bigr)W
	-
	W^*Tf\bigl(\qsym{B}\bigr)T^*W
	\le
	f\bigl(\qsym{A-B}\bigr).
	\]
	Therefore, setting
	\[
	U:=W^*,\qquad V:=W^*T,
	\]
	we get
	\[
	Uf\bigl(\qsym{A}\bigr)U^*-Vf\bigl(\qsym{B}\bigr)V^*
	\le
	f\bigl(\qsym{A-B}\bigr),
	\]
	as desired.
\end{proof}
\section{Condition-number estimates for symmetric moduli}
\label{sec:condition-number}
First, we construct two $2\times 2$ matrices to illustrate the sharpness of the coefficient of \eqref{eq:Lee}.

\medskip
\noindent\emph{Sharpness of the coefficient of \eqref{eq:Lee}.} To see that the coefficient in \eqref{eq:Lee} is best possible, it is enough to consider the case \(m=2\) and \(n=2\). Write
$
t=\sqrt{\omega},
$
in particular, \(t\ge 1\). Define
\[
A_1=\begin{pmatrix} t&0\\[1mm] 0&t^{-1}\end{pmatrix},
\qquad
R=\frac1{t^2+1}
\begin{pmatrix}
	-2t&t^2-1\\
	1-t^2&-2t
\end{pmatrix},
\qquad
J=\begin{pmatrix}0&1\\1&0\end{pmatrix},
\]
and set
\[
A_2:=J\begin{pmatrix}t^{-1}&0\\0&t\end{pmatrix}R.
\]
Since both \(J\) and \(R\) are orthogonal, the singular values of \(A_2\) are exactly \(t\) and \(t^{-1}\). The same is clearly true for \(A_1\). Hence
\[
\kappa(A_1)=\kappa(A_2)=\frac{t}{t^{-1}}=t^2=\omega.
\]

Now let
$
x=\binom{1}{-t}.
$
Since \(A_1\) is positive definite, we have \(|A_1|=A_1\). Moreover,
\[
|A_2|=R^T\begin{pmatrix}t^{-1}&0\\0&t\end{pmatrix}R.
\]
A direct calculation then shows that
\begin{equation}\label{eq:abs-sum}
\bigl(|A_1|+|A_2|\bigr)x=\frac{4t}{t^2+1}\,x.
\end{equation}

On the other hand,
\[
A_1+A_2=\frac{2}{t^2+1}
\begin{pmatrix}
	t&-t^2\\
	-1&t
\end{pmatrix},
\]
and therefore
\[
(A_1+A_2)^*(A_1+A_2)
=\frac{4}{t^2+1}
\begin{pmatrix}
	1&-t\\
	-t&t^2
\end{pmatrix}.
\]
If we denote
\[
P=\begin{pmatrix}
	1&-t\\
	-t&t^2
\end{pmatrix},
\]
then \(P^2=(t^2+1)P\), and thus
\[
|A_1+A_2|=\frac{2}{t^2+1}P.
\]
Since \(Px=(t^2+1)x\), it follows that
\begin{equation}\label{eq:sum-abs}
	|A_1+A_2|x=2x.
\end{equation}

Combining the  identities \eqref{eq:abs-sum} and \eqref{eq:sum-abs}, we obtain
\[
|A_1+A_2|x
=
\frac{t^2+1}{2t}\,\bigl(|A_1|+|A_2|\bigr)x
=
\frac{\omega+1}{2\sqrt{\omega}}\,
\bigl(|A_1|+|A_2|\bigr)x.
\]
Consequently, if one replaced the coefficient \(\frac{\omega+1}{2\sqrt{\omega}}\) in \eqref{eq:Lee} by any smaller constant \(c\), then
\[
x^*\Bigl(c\bigl(|A_1|+|A_2|\bigr)-|A_1+A_2|\Bigr)x<0,
\]
which is impossible if the corresponding operator inequality were true. This proves that the coefficient \((\omega+1)/(2\sqrt{\omega})\) is sharp.\qed

Before giving a proof of Theorem~\ref{thm:triangle-another}, we need some lemmas.
\begin{lem}\label{lem:cond-Ree}
	Let $X\in\B(\Hh)$ be positive and assume that
	\[
	\alpha I\le X\le \beta I
	\qquad
	(0<\alpha\le \beta).
	\]
	If $K\in\B(\Hh)$ is a contraction, i.e.\ $\|K\|_\infty\le 1$, then
	\[
	\Ree(KX)\le \frac{\alpha+\beta}{2\sqrt{\alpha\beta}}\,X.
	\]
\end{lem}

\begin{proof}
	Define
$
	Y:=X^{-1/2}KX^{1/2}.
$
	Then
	\[
	X^{1/2}YX^{1/2}=KX,
	\qquad
	X^{1/2}Y^*X^{1/2}=XK^*,
	\]
	and hence
	\[
	\Ree(KX)=X^{1/2}\,\Ree(Y)\,X^{1/2}.
	\]
	Therefore it suffices to show that
	\[
	\Ree(Y)\le \frac{\alpha+\beta}{2\sqrt{\alpha\beta}}\,I.
	\]
	
	Let $x\in\Hh$ be a unit vector. Then
	\[
	\langle \Ree(Y)x,x\rangle
	=
	\Ree\langle Yx,x\rangle
	\le
	|\langle Yx,x\rangle|.
	\]
	By the Cauchy--Schwarz inequality and the assumption $\|K\|_\infty\le 1$,
	\[
	|\langle Yx,x\rangle|
	=
	\bigl|\langle KX^{1/2}x,X^{-1/2}x\rangle\bigr|
	\le
	\|KX^{1/2}x\|_2\,\|X^{-1/2}x\|_2
	\le
	\|X^{1/2}x\|_2\,\|X^{-1/2}x\|_2.
	\]
	
	On the other hand, for every $t\in[\alpha,\beta]$ one has
	\[
	(t-\alpha)(\beta-t)\ge 0,
	\]
	which is equivalent to
	\[
	t+\alpha\beta\,t^{-1}\le \alpha+\beta.
	\]
	By functional calculus,
	\[
	X+\alpha\beta\,X^{-1}\le (\alpha+\beta)I.
	\]
	Taking inner products with $x$ gives
	\[
	\langle Xx,x\rangle+\alpha\beta\,\langle X^{-1}x,x\rangle\le \alpha+\beta.
	\]
	Using the arithmetic--geometric mean inequality,
	\[
	2\sqrt{\alpha\beta}\,
	\sqrt{\langle Xx,x\rangle\langle X^{-1}x,x\rangle}
	\le
	\langle Xx,x\rangle+\alpha\beta\,\langle X^{-1}x,x\rangle
	\le
	\alpha+\beta.
	\]
	Hence
	\[
	\|X^{1/2}x\|_2\,\|X^{-1/2}x\|_2
	=
	\sqrt{\langle Xx,x\rangle\langle X^{-1}x,x\rangle}
	\le
	\frac{\alpha+\beta}{2\sqrt{\alpha\beta}}.
	\]
	Therefore
	\[
	\langle \Ree(Y)x,x\rangle\le \frac{\alpha+\beta}{2\sqrt{\alpha\beta}}.
	\]
	Since this holds for every unit vector $x$, we obtain
	\[
	\Ree(Y)\le \frac{\alpha+\beta}{2\sqrt{\alpha\beta}}\,I.
	\]
	Multiplying on both sides by $X^{1/2}$ yields
	\[
	\Ree(KX)\le \frac{\alpha+\beta}{2\sqrt{\alpha\beta}}\,X.
	\]
\end{proof}

\begin{prop}\label{prop:cond-rect}
	Let $A_1,\dots,A_m\in\B(\Hh,\Kk)$.
	Assume that for each $i$ there exist real numbers $0<\alpha_i\le \beta_i$ such that
	\[
	\alpha_i I\le \absop{A_i}\le \beta_i I.
	\]
	Then
	\[
	\absop{\sum_{i=1}^mA_i}
	\le
\sum_{i=1}^m
\frac{\alpha_i+\beta_i}{2\sqrt{\alpha_i\beta_i}}\,\absop{A_i}.
	\]
\end{prop}

\begin{proof}
	Set
$
	S:=\sum_{i=1}^mA_i
$
	and write the polar decomposition
	\[
	S=W\absop{S},
	\]
	where $W\in\B(\Hh,\Kk)$ is a partial isometry.
	Then
	\[
	W^*S=W^*W\absop{S}=\absop{S}.
	\]
	
	For each $i$, write
$
	A_i=U_i\absop{A_i}
$
	for the polar decomposition of $A_i$.
	Since $\absop{A_i}\ge \alpha_i I>0$, the support projection of $\absop{A_i}$ is $I$,
	and therefore $U_i^*U_i=I$; that is, each $U_i$ is an isometry.
	In particular, $W^*U_i$ is a contraction.
	
	Now
	\[
	\absop{S}=W^*S=\sum_{i=1}^m W^*U_i\absop{A_i}.
	\]
	Since $\absop{S}$ is self-adjoint, taking real parts gives
	\[
	\absop{S}
	=
	\sum_{i=1}^m \Ree\,\!\bigl(W^*U_i\absop{A_i}\bigr).
	\]
	Applying Lemma~\ref{lem:cond-Ree} to
	\[
	X=\absop{A_i},\qquad K=W^*U_i,
	\]
	we obtain
	\[
	\Ree\!\bigl(W^*U_i\absop{A_i}\bigr)
	\le
	\frac{\alpha_i+\beta_i}{2\sqrt{\alpha_i\beta_i}}\,\absop{A_i}
	\]
	for every $i$.
	Summing these inequalities yields
	\[
	\absop{S}
	\le
	\sum_{i=1}^m
	\frac{\alpha_i+\beta_i}{2\sqrt{\alpha_i\beta_i}}\,\absop{A_i}.
	\]
\end{proof}
\begin{lem}\label{lem:one-sided-sharp}
	Let
$
	c(\kappa):=\frac{\kappa+1}{2\sqrt{\kappa}}
$ with $\kappa\ge 1.$
	Then for  every \(a<c(\kappa)\), there exist invertible
	\(2\times 2\) matrices \(A,B\) such that
	\[
	\kappa(A)=\kappa,\qquad \kappa(B)=1,
	\]
	but
	\[
	|A+B|\nleq a\,|A|+|B|.
	\]
	In particular, the coefficient \(c(\kappa)\) in front of \(|A|\) is sharp even when
	the coefficient in front of \(|B|\) is fixed at \(1=c(1)\).
\end{lem}

\begin{proof}
	If \(\kappa=1\), then \(c(1)=1\). Let \(a<1\), take
	\[
	A=I_2,\qquad B=\varepsilon I_2
	\]
	with \(0<\varepsilon<1-a\). Then \(\kappa(A)=\kappa(B)=1\), while
	\[
	|A+B|=(1+\varepsilon)I_2 \nleq (a+\varepsilon)I_2=a|A|+|B|.
	\]
	So the claim holds when \(\kappa=1\).
	
	Assume now that \(\kappa>1\), and write \(t:=\sqrt{\kappa}>1\). Set
	\[
	m:=\frac{t+t^{-1}}{2}=c(\kappa),
	\qquad
	d:=\frac{t-t^{-1}}{2}.
	\]
	Define
	\[
	A:=
	\begin{pmatrix}
		0 & t^{-1}\\
		t & 0
	\end{pmatrix},
	\qquad
	B_\mu:=
	\begin{pmatrix}
		\mu & d\\
		-d & \mu
	\end{pmatrix},
	\qquad \mu\ge m.
	\]
	Then
	\[
	A^*A=
	\begin{pmatrix}
		t^2 & 0\\
		0 & t^{-2}
	\end{pmatrix},
	\qquad
	|A|=
	\begin{pmatrix}
		t & 0\\
		0 & t^{-1}
	\end{pmatrix},
	\]
	so \(\kappa(A)=t^2=\kappa\).
	
	Also,
	\[
	B_\mu^*B_\mu=(\mu^2+d^2)I_2,
	\]
	hence
	\[
	|B_\mu|=\sqrt{\mu^2+d^2}\,I_2,
	\qquad
	\kappa(B_\mu)=1.
	\]
	Moreover,
	\[
	A+B_\mu=
	\begin{pmatrix}
		\mu & m\\
		m & \mu
	\end{pmatrix}\ge 0
	\qquad (\mu\ge m),
	\]
	so
	\[
	|A+B_\mu|=A+B_\mu=
	\begin{pmatrix}
		\mu & m\\
		m & \mu
	\end{pmatrix}.
	\]
	
	Suppose that
	\[
	|A+B_\mu|\le a\,|A|+|B_\mu|.
	\]
	Writing
	\[
	r_\mu:=\sqrt{\mu^2+d^2},
	\qquad
	\delta_\mu:=r_\mu-\mu>0,
	\]
	this becomes
	\[
	\begin{pmatrix}
		at+\delta_\mu & -m\\
		-m & at^{-1}+\delta_\mu
	\end{pmatrix}\ge 0.
	\]
	Hence its determinant must be nonnegative:
	\[
	(at+\delta_\mu)(at^{-1}+\delta_\mu)-m^2\ge 0.
	\]
	Since \(tt^{-1}=1\) and \(t+t^{-1}=2m\), this is
	\[
	a^2+2m\delta_\mu\,a+\delta_\mu^2-m^2\ge 0.
	\]
	Therefore
	\[
	a\ge \rho_\mu:=
	-m\delta_\mu+\sqrt{m^2+(m^2-1)\delta_\mu^2}.
	\]
	Now \(\delta_\mu=r_\mu-\mu\to 0\) as \(\mu\to\infty\), so
	\[
	\rho_\mu\longrightarrow m=c(\kappa).
	\]
	Given any \(a<c(\kappa)\), we can choose \(\mu\) sufficiently large so that
	\[
	a<\rho_\mu.
	\]
	For this choice of \(\mu\), the inequality
	\[
	|A+B_\mu|\le a\,|A|+|B_\mu|
	\]
	cannot hold. This proves the lemma.
\end{proof}
Next, we give a proof of Theorem~\ref{thm:triangle-another} based on Proposition~\ref{prop:cond-rect} and Lemma~\ref{lem:one-sided-sharp}.
\begin{proof}[Proof of Theorem~\ref{thm:triangle-another}]
	For each $i$, set
	\[
	\alpha_i:=\|A_i^{-1}\|_{\infty}^{-1},
	\qquad
	\beta_i:=\|A_i\|_{\infty}.
	\]
	Then
	\[
	\alpha_i I\le \absop{A_i}\le \beta_i I,
	\qquad
	\frac{\beta_i}{\alpha_i}=\|A_i\|_{\infty}\,\|A_i^{-1}\|_{\infty}=\kappa_i.
	\]
	Also,
$
	\alpha_i I\le \absop{A_i^*}\le \beta_i I,
$
because $\|A_i^*\|_{\infty}=\|A_i\|_{\infty}$ and $\|(A_i^*)^{-1}\|_{\infty}=\|A_i^{-1}\|_{\infty}$.

	By Proposition~\ref{prop:cond-rect},
	\[
	\absop{A_1+\cdots+A_m}
	\le
		\frac{\kappa_1+1}{2\sqrt{\kappa_1}}\absop{A_1}+\cdots+	\frac{\kappa_m+1}{2\sqrt{\kappa_m}}\absop{A_m},
	\]
	and similarly,
	\[
	\absop{A_1^*+\cdots+A_m^*}
	\le
\frac{\kappa_1+1}{2\sqrt{\kappa_1}}\absop{A_1^*}+\cdots+\frac{\kappa_m+1}{2\sqrt{\kappa_m}}\absop{A_m^*}.
	\]
	Thus,
\[
\sym{A_1+\cdots+A_m}
\le
\frac{\kappa_1+1}{2\sqrt{\kappa_1}}\,\sym{A_1}
+\cdots+
\frac{\kappa_m+1}{2\sqrt{\kappa_m}}\,\sym{A_m}.
\]

Next, recall the lifting
\[
\Gamma:\B(\Hh)\to\B(\Hh,\Hh\oplus\Hh),
\qquad
\Gamma(T):=\frac{1}{\sqrt2}\binom{T}{T^*}.
\]
Then $\Gamma$ is linear and, for each $i$,
\[
\absop{\Gamma(A_i)}^2
=
\qsym{A_i}^2
=
\frac{\absop{A_i}^2+\absop{A_i^*}^2}{2}.
\]
From
\[
\alpha_i I\le \absop{A_i},\ \absop{A_i^*}\le \beta_i I,
\]
it follows that
\[
\alpha_i I\le \absop{\Gamma(A_i)}\le \beta_i I.
\]
Applying Proposition~\ref{prop:cond-rect} to
\[
\Gamma(A_1),\dots,\Gamma(A_m)\in\B(\Hh,\Hh\oplus\Hh),
\]
yields
\[
\absop{\Gamma(A_1+\cdots+A_m)}
\le
\frac{\kappa_1+1}{2\sqrt{\kappa_1}}\,\absop{\Gamma(A_1)}
+\cdots+
\frac{\kappa_m+1}{2\sqrt{\kappa_m}}\,\absop{\Gamma(A_m)}.
\]
That is,
\[
\qsym{A_1+\cdots+A_m}
\le
\frac{\kappa_1+1}{2\sqrt{\kappa_1}}\,\qsym{A_1}
+\cdots+
\frac{\kappa_m+1}{2\sqrt{\kappa_m}}\,\qsym{A_m}.
\]
	This completes the proof.

We now prove the sharpness statement in Theorem~\ref{thm:triangle-another}.
Fix \(\kappa\ge 1\), and let \(c(\kappa)=(\kappa+1)/(2\sqrt{\kappa})\).
By Lemma~\ref{lem:one-sided-sharp}, for every \(a<c(\kappa)\) there exist invertible
\(2\times 2\) matrices \(A,B\) such that
\[
\kappa(A)=\kappa,\qquad \kappa(B)=1,
\]
but
\[
|A+B|\nleq a\,|A|+|B|.
\]
Thus the coefficient \(c(\kappa)\) is already sharp in the two-variable inequality for the
usual modulus.

To transfer this sharpness to the arithmetic and quadratic symmetric moduli, consider the
self-adjoint block matrices
\[
\widehat A:=
\begin{pmatrix}
	0 & A\\
	A^* & 0
\end{pmatrix},
\qquad
\widehat B:=
\begin{pmatrix}
	0 & B\\
	B^* & 0
\end{pmatrix}.
\]
Then
\[
\widehat A=\widehat A^*,\qquad \widehat B=\widehat B^*,
\]
and
\[
\kappa(\widehat A)=\kappa(A)=\kappa,
\qquad
\kappa(\widehat B)=\kappa(B)=1.
\]
Since \(\widehat A\) and \(\widehat B\) are self-adjoint, all three moduli coincide:
\[
\sym{\widehat A}=\qsym{\widehat A}=|\widehat A|
=
\begin{pmatrix}
	|A^*| & 0\\
	0 & |A|
\end{pmatrix},
\qquad
\sym{\widehat B}=\qsym{\widehat B}=|\widehat B|
=
\begin{pmatrix}
	|B^*| & 0\\
	0 & |B|
\end{pmatrix},
\]
and similarly
\[
\sym{\widehat A+\widehat B}
=
\qsym{\widehat A+\widehat B}
=
|\widehat A+\widehat B|
=
\begin{pmatrix}
	|A^*+B^*| & 0\\
	0 & |A+B|
\end{pmatrix}.
\]

Suppose, for contradiction, that one could replace the coefficient \(c(\kappa)\) in front of
the first term by some smaller constant \(a<c(\kappa)\) in the arithmetic symmetric modulus
inequality. Applying that putative inequality to \(\widehat A,\widehat B\) would give
\[
\sym{\widehat A+\widehat B}
\le
a\,\sym{\widehat A}+\sym{\widehat B}.
\]
Comparing the lower-right block yields
\[
|A+B|\le a\,|A|+|B|,
\]
contrary to Lemma~\ref{lem:one-sided-sharp}. Hence the coefficient \(c(\kappa)\) is sharp for
\(\sym{\cdot}\).

Exactly the same block comparison shows that \(c(\kappa)\) is also sharp for
\(\qsym{\cdot}\).

Consequently, the coefficient function
\[
c(\kappa)=\frac{\kappa+1}{2\sqrt{\kappa}}
\]
is individually sharp for each of the three moduli
\[
|\cdot|,\qquad |\cdot|_{\mathrm{sym}},\qquad |\cdot|_{\mathrm{qsym}}.
\]
This completes the proof.
\end{proof}
\section{Clarkson--McCarthy type inequalities for the quadratic symmetric modulus}
\label{sec:clarkson}
\begin{proof}[Proof of Theorem~\ref{thm:identity}]
	Recall that for every operator $T$,
	\[
	|T|_{\mathrm{qsym}}^{2}
	=
	\frac{|T|^2+|T^*|^2}{2}
	=
	\frac{T^*T+TT^*}{2}.
	\]
	For each $1\le i\le m$, set
$
	T_i:=\sum_{j=1}^n u_{ij}A_j.
$
	Then
	\[
	\sum_{i=1}^m |T_i|_{\mathrm{qsym}}^{2}
	=
	\frac12\sum_{i=1}^m\bigl(T_i^*T_i+T_iT_i^*\bigr).
	\]
	
	Assume first that $U^*U=I_n$.
	Then
	\[
	\sum_{i=1}^m T_i^*T_i
	=
	\sum_{i=1}^m
	\left(\sum_{j=1}^n \overline{u_{ij}}A_j^*\right)
	\left(\sum_{k=1}^n u_{ik}A_k\right)
	=
	\sum_{j,k=1}^n
	\left(\sum_{i=1}^m \overline{u_{ij}}u_{ik}\right)A_j^*A_k.
	\]
	Since
	\[
	\sum_{i=1}^m \overline{u_{ij}}u_{ik}=(U^*U)_{jk}=\delta_{jk},
	\]
	we obtain
	\[
	\sum_{i=1}^m T_i^*T_i=\sum_{j=1}^n A_j^*A_j.
	\]
	Similarly,
	\[
	\sum_{i=1}^m T_iT_i^*
	=
	\sum_{i=1}^m
	\left(\sum_{j=1}^n u_{ij}A_j\right)
	\left(\sum_{k=1}^n \overline{u_{ik}}A_k^*\right)
	=
	\sum_{j,k=1}^n
	\left(\sum_{i=1}^m u_{ij}\overline{u_{ik}}\right)A_jA_k^*
	=
	\sum_{j=1}^n A_jA_j^*.
	\]
	Therefore
	\[
	\sum_{i=1}^m |T_i|_{\mathrm{qsym}}^{2}
	=
	\frac12\sum_{j=1}^n \bigl(A_j^*A_j+A_jA_j^*\bigr)
	=
	\sum_{j=1}^n |A_j|_{\mathrm{qsym}}^{2}.
	\]
	
	Conversely, assume that
	\[
	\sum_{i=1}^m
	\left|
	\sum_{j=1}^n u_{ij}A_j
	\right|_{\mathrm{qsym}}^{2}
	=
	\sum_{j=1}^n |A_j|_{\mathrm{qsym}}^{2}
	\]
	holds for all $A_1,\ldots,A_n\in \mathbb{M}_N$.
	Let $P$ be any rank-one projection in $\mathbb{M}_N$, and choose arbitrary scalars
	$a_1,\ldots,a_n\in\mathbb C$. Put
	\[
	A_j:=a_jP
	\qquad (1\le j\le n).
	\]
	Since $P=P^*=P^2$, we have
	\[
	\left|
	\left(\sum_{j=1}^n u_{ij}a_j\right)P
	\right|_{\mathrm{qsym}}^{2}
	=
	\left|
	\sum_{j=1}^n u_{ij}a_j
	\right|^2 P,
	\qquad
	|a_jP|_{\mathrm{qsym}}^{2}=|a_j|^2P.
	\]
	Hence the assumed identity becomes
	\[
	\sum_{i=1}^m
	\left|
	\sum_{j=1}^n u_{ij}a_j
	\right|^2 P
	=
	\sum_{j=1}^n |a_j|^2 P.
	\]
	Since $P\ne 0$, it follows that
	\[
	\sum_{i=1}^m
	\left|
	\sum_{j=1}^n u_{ij}a_j
	\right|^2
	=
	\sum_{j=1}^n |a_j|^2
	\qquad
	\text{for all }(a_1,\ldots,a_n)\in\mathbb C^n.
	\]
	This is exactly the condition that $U$ is an isometry, namely,
$
	U^*U=I_n.
$
	The proof is complete.
\end{proof}
	We now turn to a unitary-orbit version of Jensen's inequality for monotone convex/concave functions.
These orbit inequalities yield matrix-order dominations, and hence  deduce Schatten $p$-norm consequences.

The following unitary-orbit inequalities on monotone concave/convex functions will be used in the sequel.
Lemma~\ref{lem:unitary_orbit_conv_conc}\,(1) \emph{the first inequality} is obtained by a straightforward induction on $n$
from the corresponding two-variable inequality \cite[Corollary~3.2]{BL12}.
Lemma~\ref{lem:unitary_orbit_conv_conc}\,(2) \emph{the first inequality} follows similarly from
\cite[Theorem~3.1]{BL12} (equivalently, \cite[Theorem~2.1]{AB07}).
Lemma~\ref{lem:unitary_orbit_conv_conc}\,(1) \emph{the second inequality} follows from the
isometric-column Jensen inequality \cite[Corollary~2.4]{BL12} by choosing $Z_i=\sqrt{\alpha_i}\,I$,
so that $\sum_{i=1}^n Z_i^*Z_i=I$ and $\sum_{i=1}^n Z_i^*A_iZ_i=\sum_{i=1}^n \alpha_i A_i$.
Similarly, Lemma~\ref{lem:unitary_orbit_conv_conc}\,(2) \emph{the second inequality} follows from the same result,
using the fact that the inequality in \cite[Corollary~2.4]{BL12} reverses for concave functions.

\begin{lem}\label{lem:unitary_orbit_conv_conc}
	Let $A_1,\dots,A_n\in\Mn$ be positive semidefinite and let
	$\alpha_1,\dots,\alpha_n$ be positive real numbers such that $\sum_{j=1}^n\alpha_j=1$.
	Then 
	\begin{enumerate}
		\item for every monotone convex function $f:[0,\infty)\to\mathbb{R}$ with $f(0)\le 0$,
		there exist unitaries $U,U_1,\dots,U_n\in\Mn$ such that
		\begin{align*}
			f\!\Bigl(\sum_{j=1}^n A_j\Bigr)
			\ &\ge\
			\sum_{j=1}^n U_j f(A_j)U_j^*,\\
			f\!\Bigl(\sum_{j=1}^n \alpha_j A_j\Bigr)
			\ &\le\
			U\Bigl(\sum_{j=1}^n \alpha_j f(A_j)\Bigr)U^*;
		\end{align*}
		
		\item for every monotone concave function $f:[0,\infty)\to\mathbb{R}$ with $f(0)\ge 0$,
		there exist unitaries $U,U_1,\dots,U_n\in\Mn$ such that
		\begin{align*}
			f\!\Bigl(\sum_{j=1}^n A_j\Bigr)
			\ &\le\
			\sum_{j=1}^n U_j f(A_j)U_j^*,\\
			f\!\Bigl(\sum_{j=1}^n \alpha_j A_j\Bigr)
			\ &\ge\
			U\Bigl(\sum_{j=1}^n \alpha_j f(A_j)\Bigr)U^*.
		\end{align*}
	\end{enumerate}
\end{lem}
	\begin{proof}[Proof of Theorem~\ref{thm:unitary-orbit}]
		By Theorem~\ref{thm:identity}, the assumption implies that
		\[
		\sum_{i=1}^m |B_i|_{\mathrm{qsym}}^{2}
		=
		\sum_{j=1}^n |A_j|_{\mathrm{qsym}}^{2}.
		\]
		
		\noindent\emph{Case 1: \(p\ge2\).}
		Take \(f(t)=t^{p/2}\) in Lemma~\ref{lem:unitary_orbit_conv_conc}~(1).
		Then there exist unitaries \(U_1,\ldots,U_m\) and \(W\) such that
		\begin{align*}
			\sum_{i=1}^m U_i|B_i|_{\mathrm{qsym}}^pU_i^*
			&\le
			\left(\sum_{i=1}^m |B_i|_{\mathrm{qsym}}^{2}\right)^{p/2}\\
			&=
			\left(\sum_{j=1}^n |A_j|_{\mathrm{qsym}}^{2}\right)^{p/2}\\
			&=
			n^{p/2}
			\left(\frac1n\sum_{j=1}^n |A_j|_{\mathrm{qsym}}^{2}\right)^{p/2}\\
			&\le
			n^{\frac p2-1}
			W\left(\sum_{j=1}^n |A_j|_{\mathrm{qsym}}^p\right)W^*.
		\end{align*}
		Conjugating by \(W^*\) and replacing each \(U_i\) by \(W^*U_i\), we obtain
		\[
		\sum_{i=1}^m U_i|B_i|_{\mathrm{qsym}}^pU_i^*
		\le
		n^{\frac p2-1}\sum_{j=1}^n |A_j|_{\mathrm{qsym}}^p.
		\]
		
		The second inequality is proved in the same way, interchanging the roles of
		\((A_1,\ldots,A_n)\) and \((B_1,\ldots,B_m)\).
		
		\medskip
		\noindent	\emph{Case 2: \(0<p\le2\).}
		The argument is the same, except that the inequalities in Lemma~\ref{lem:unitary_orbit_conv_conc} reverse.
	\end{proof}

Next,	we illustrate a Jensen-type principle for unitarily invariant norms, which allows us to convert the square-sum
identity \eqref{eq:identity} into $p$-power estimates via the function $t\mapsto t^{p/2}$.

We will use the following Jensen-type inequalities for unitarily invariant norms.
\begin{lem}[{\cite[Lemma~2.1]{HK08}}]\label{lem:Jensen_type_inequalities_uin}
	Let $m\ge1$. Let $A_1,\dots,A_m\in\Mn$ be positive semidefinite and let
	$\alpha_1,\dots,\alpha_m$ be positive real numbers such that $\sum_{j=1}^m \alpha_j=1$.
	Then
	\begin{enumerate}
		\item for every convex function $f:[0,\infty)\to[0,\infty)$ with $f(0)=0$,
		\begin{align*}
			\unorm{\, f\!\Bigl(\sum_{j=1}^m \alpha_j A_j\Bigr)\,}
			&\le
			\unorm{\, \sum_{j=1}^m \alpha_j f(A_j)\,},\\
			\unorm{\, \sum_{j=1}^m f(A_j)\,}
			&\le
			\unorm{\, f\!\Bigl(\sum_{j=1}^m A_j\Bigr)\,};
		\end{align*}
		
		\item for every concave function $f:[0,\infty)\to[0,\infty)$ with $f(0)=0$,
		\begin{align*}
			\unorm{\, \sum_{j=1}^m \alpha_j f(A_j)\,}
			&\le
			\unorm{\, f\!\Bigl(\sum_{j=1}^m \alpha_j A_j\Bigr)\,},\\
			\unorm{\, f\!\Bigl(\sum_{j=1}^m A_j\Bigr)\,}
			&\le
			\unorm{\, \sum_{j=1}^m f(A_j)\,}.
		\end{align*}
	\end{enumerate}
\end{lem}
\begin{proof}[Proof of Theorem~\ref{thm:uin}] By Theorem~\ref{thm:identity}, the condition $(B_1,\ldots,B_m)^T=U(A_1,\ldots,A_n)^T$ for some isometry $U=(u_{ij})$ implies that
	\[
			\sum_{i=1}^m
	\left|
 B_i
	\right|_{\mathrm{qsym}}^{2}
	=
	\sum_{j=1}^n |A_j|_{\mathrm{qsym}}^{2}.
	\]
\emph{Case 1: $p\ge 2$.} Take $f(t)=t^{\frac{p}{2}}$ in Lemma~\ref{lem:Jensen_type_inequalities_uin}~(1), we have
\begin{align*}
	\unorm{\sum_{i=1}^m
		\left|
		B_i
		\right|_{\mathrm{qsym}}^{p}}&\le \unorm{\left( \sum_{i=1}^m
		\left|
		B_i
		\right|_{\mathrm{qsym}}^{2}\right)^\frac{p}{2}}\\
		&=\unorm{\left( \sum_{j=1}^n
			\left|
		A_j
			\right|_{\mathrm{qsym}}^{2}\right)^\frac{p}{2}}
			=n^{\frac{p}{2}}\unorm{\left( \sum_{j=1}^n
				\left|
				A_j
				\right|_{\mathrm{qsym}}^{2}/n\right)^\frac{p}{2}}\\
				&\le n^{\frac{p}{2}-1}\unorm{\sum_{j=1}^n |A_j|_{\mathrm{qsym}}^{p}}.
\end{align*}
Similarly, we have
\[
\unorm{\sum_{j=1}^n |A_j|_{\mathrm{qsym}}^{p}}\le m^{\frac{p}{2}-1}\unorm{\sum_{i=1}^m
	\left|
	B_i
	\right|_{\mathrm{qsym}}^{p}}.
\]
\emph{Case 2: $0<p\le 2$.} The argument is analogous to that for the case \(p \ge 2\), except that all the inequalities are reversed.
\end{proof}
\section{Proof of Theorem~\ref{thm:Q23-all-x-outside}}
\label{sec:proof-q23}
	\begin{proof}[Proof of Theorem~\ref{thm:Q23-all-x-outside}]
	Let $n=1$ and take scalar matrices
	\[
	A=1,\qquad B=\frac{x-1}{x}\quad (\text{note that }x\neq 0 \text{ since }x\notin[0,1]).
	\]
	Then
	\[
	A\nabla_x B=(1-x)A+xB=(1-x)+x\cdot\frac{x-1}{x}=0,
	\]
	while
	\[
	B\nabla_x A=(1-x)B+xA=(1-x)\frac{x-1}{x}+x=\frac{2x-1}{x}\neq 0,
	\]
	because $x\notin[0,1]$ implies $x\neq \tfrac12$.
	Moreover,
	\[
	A-B=1-\frac{x-1}{x}=\frac1x\neq 0.
	\]
	
	Assume, for contradiction, that Theorem~\ref{thm:BL-C22} holds for this $x$.
	Since $x\notin[0,1]$, we have $x(1-x)<0$.
	With $A\nabla_x B=0$, the identity in Theorem~\ref{thm:BL-C22} would reduce to
	\[
	|A\oplus B|^{2}
	=
	V\,|B\nabla_x A|^{2}V^{*}
	+
	x(1-x)\Bigl\{\,S\,|A-B|^{2}S^{*}+T\,|A-B|^{2}T^{*}\Bigr\},
	\]
	for some isometries $V,S,T\in \mathbb M_{2,1}$.
	Moving the last term to the left and using $-x(1-x)>0$, we obtain
	\[
	|A\oplus B|^{2}
	+(-x(1-x))\Bigl\{\,S\,|A-B|^{2}S^{*}+T\,|A-B|^{2}T^{*}\Bigr\}
	=
	V\,|B\nabla_x A|^{2}V^{*}.
	\]
	Here $|A\oplus B|^{2}=\diag(|A|^{2},|B|^{2})=\diag\!\bigl(1,|(x-1)/x|^{2}\bigr)$ is positive definite,
	and the bracketed term is positive semidefinite; hence the whole left-hand side is positive definite,
	thus has rank $2$.
	
	On the other hand, since $V\in M_{2,1}$ is an isometry, $VV^{*}$ is a rank-one projection; moreover
	$|B\nabla_x A|^{2}>0$. Therefore the right-hand side $V\,|B\nabla_x A|^{2}V^{*}$ has rank $1$.
	
	This is impossible. Hence Theorem~\ref{thm:BL-C22} cannot hold for this $x$.
	Since $x\in\mathbb R\setminus[0,1]$ was arbitrary, the claim follows.
\end{proof}
\section{Proof of Theorem~\ref{thm:main1}}
\label{sec:proof-main1}
In this section, we prove Theorem~\ref{thm:main1}.

\subsection{Trace necessary condition}

Assume \eqref{eq:Bourin_Lee_p_le_2} holds for a given $Z\in\Mn$ and some unitaries $U,V$. Then, taking traces yields
\begin{equation}\label{eq:trace-necessary}
	\Tr\left(\frac{|Z|^{p}+|Z^{*}|^{p}}{2}\right)^{1/p}
	\le \Tr|\Ree\,Z|+\Tr|\Ima\,Z|.
\end{equation}
Indeed, all terms in \eqref{eq:Bourin_Lee_p_le_2} are positive semidefinite, and for $0\le A\le B$ we have $\Tr A\le \Tr B$.
Moreover, the right-hand side of \eqref{eq:Bourin_Lee_p_le_2} has trace independent of the choice of $U,V$ by unitary invariance.
Consequently, if \eqref{eq:trace-necessary} fails for some $Z$, then \eqref{eq:Bourin_Lee_p_le_2} fails
for that $Z$ for all choices of $U,V$.

Thus, to disprove \eqref{eq:Bourin_Lee_p_le_2} for $p>2$, it suffices to find, for a given $p$, a matrix $Z$ violating
\eqref{eq:trace-necessary}.

\subsection{A $2\times 2$ counterexample for $p>2$}

Fix $\theta\in(0,\pi/2)$ and set $c=\cos\theta$, $s=\sin\theta$. Define
\begin{equation}\label{eq:Ztheta}
	Z_\theta=
	\begin{pmatrix}
		c & 0\\
		-s & 0
	\end{pmatrix}\in\mathbb{M}_2.
\end{equation}
Note that $Z_\theta$ is real, hence $Z_\theta^*=Z_\theta^{\mathsf T}$.

\begin{cla}\label{cla:left}
	For every $p>0$,
	\begin{equation}\label{eq:left-trace}
		\Tr\left(\frac{|Z_\theta|^{p}+|Z_\theta^{*}|^{p}}{2}\right)^{1/p}
		=\left(\frac{1+c}{2}\right)^{1/p}+\left(\frac{1-c}{2}\right)^{1/p}.
	\end{equation}
\end{cla}

\begin{proof}
	Compute
	\[
	Z_\theta^*Z_\theta=
	\begin{pmatrix}
		c&-s\\ 0&0
	\end{pmatrix}
	\begin{pmatrix}
	c & 0\\
	-s & 0
\end{pmatrix}
	=
	\begin{pmatrix}
		1 & 0\\
		0 & 0
	\end{pmatrix}
	=:P.
	\]
	Hence $|Z_\theta|=(Z_\theta^*Z_\theta)^{1/2}=P$. Similarly,
	\[
	Z_\theta Z_\theta^*=
	\begin{pmatrix}
		c^2 & -cs\\
		-cs & s^2
	\end{pmatrix}
	=:Q.
	\]
	A direct computation shows $Q^2=Q$. Moreover $Q=Q^*$ (indeed $Q$ is real symmetric), hence $Q$ is an orthogonal projection.
	Therefore $|Z_\theta^*|=(Z_\theta Z_\theta^*)^{1/2}=Q$.
	
	Since $P$ and $Q$ are projections, $P^p=P$ and $Q^p=Q$ for every $p>0$.
	Thus
	\[
	\left(\frac{|Z_\theta|^{p}+|Z_\theta^{*}|^{p}}{2}\right)^{1/p}
	=
	\left(\frac{P+Q}{2}\right)^{1/p}.
	\]
	We compute the eigenvalues of $(P+Q)/2$. We have
	\[
	P+Q=
	\begin{pmatrix}
		1+c^2 & -cs\\
		-cs & s^2
	\end{pmatrix},
	\qquad
	\Tr(P+Q)=2,
	\qquad
	\det(P+Q)=s^2.
	\]
	Hence the characteristic polynomial is $\lambda^2-2\lambda+s^2=0$, so the eigenvalues are
	$\lambda_\pm=1\pm\sqrt{1-s^2}=1\pm c$.
	Therefore the eigenvalues of $(P+Q)/2$ are $(1\pm c)/2$, and those of $\bigl((P+Q)/2\bigr)^{1/p}$
	are $\bigl((1\pm c)/2\bigr)^{1/p}$.
	Taking the trace yields \eqref{eq:left-trace}.
\end{proof}

\begin{cla}\label{cla:right}
	For $Z_\theta$ in \eqref{eq:Ztheta},
	\begin{equation}\label{eq:right-trace}
		\Tr|\Ree\,Z_\theta|+\Tr|\Ima\,Z_\theta| = 1+s.
	\end{equation}
\end{cla}

\begin{proof}
	We compute
	\[
	\Ree\,Z_\theta=\frac{Z_\theta+Z_\theta^*}{2}
	=
	\begin{pmatrix}
		c & -s/2\\
		-s/2 & 0
	\end{pmatrix},
	\qquad
	\Ima\,Z_\theta=\frac{Z_\theta-Z_\theta^*}{2i}
	=
	\begin{pmatrix}
		0 & -is/2\\
		is/2 & 0
	\end{pmatrix}.
	\]
	For $\Ree\,Z_\theta$, the characteristic polynomial is
	\[
	\lambda^2-c\lambda-\frac{s^2}{4},
	\]
	whose roots are
	\[
	\lambda_\pm=\frac{c\pm\sqrt{c^2+s^2}}{2}=\frac{c\pm 1}{2}.
	\]
	Since $c\in(0,1)$, we have $\lambda_+>0$ and $\lambda_-<0$, hence
	\[
	\Tr|\Ree\,Z_\theta|
	=|\lambda_+|+|\lambda_-|
	=\frac{1+c}{2}+\frac{1-c}{2}=1.
	\]
	For $\Ima\,Z_\theta$, the eigenvalues are $\pm s/2$, hence $\Tr|\Ima\,Z_\theta|=s$.
	Adding yields \eqref{eq:right-trace}.
\end{proof}

Combining \eqref{eq:trace-necessary}, Claims \ref{cla:left} and \ref{cla:right}, we see that if
\eqref{eq:Bourin_Lee_p_le_2} were to hold for $Z_\theta$, then necessarily
\begin{equation}\label{eq:necessary-ineq-theta}
	\left(\frac{1+c}{2}\right)^{1/p}+\left(\frac{1-c}{2}\right)^{1/p} \le 1+s.
\end{equation}
\begin{cla}\label{cla:asymptotic}
	Let $p>2$. Then there exists $\theta_0\in(0,\pi/2)$ such that for all
	$\theta\in(0,\theta_0)$ the inequality \eqref{eq:necessary-ineq-theta} fails.
\end{cla}

\begin{proof}
	Define
	\[
	\Phi(\theta):=\left(\frac{1+\cos\theta}{2}\right)^{1/p}+\left(\frac{1-\cos\theta}{2}\right)^{1/p}-(1+\sin\theta).
	\]
	As $\theta\to0^+$ we have the standard expansions
	\[
	\sin\theta=\theta+O(\theta^3),\qquad
	\cos\theta=1-\frac{\theta^2}{2}+O(\theta^4).
	\]
	Hence
	\[
	\frac{1+\cos\theta}{2}=1-\frac{\theta^2}{4}+O(\theta^4)
	\quad\Longrightarrow\quad
	\left(\frac{1+\cos\theta}{2}\right)^{1/p}=1+O(\theta^2),
	\]
	and
	\begin{align*}
		&\frac{1-\cos\theta}{2}=\frac{\theta^2}{4}+O(\theta^4)
		=\frac{\theta^2}{4}\bigl(1+O(\theta^2)\bigr)\\
		&\Longrightarrow
		\left(\frac{1-\cos\theta}{2}\right)^{1/p}
		=4^{-1/p}\theta^{2/p}\bigl(1+O(\theta^2)\bigr)
		=4^{-1/p}\theta^{2/p}+O(\theta^{2/p+2}).
	\end{align*}

	Combining these estimates gives
	\[
	\Phi(\theta)
	=4^{-1/p}\theta^{2/p}-\theta+O(\theta^2)+O(\theta^{2/p+2})
	\qquad(\theta\to0^+).
	\]
	Dividing by $\theta$ yields
	\[
	\frac{\Phi(\theta)}{\theta}
	=
	4^{-1/p}\theta^{2/p-1}-1+O(\theta)+O(\theta^{2/p+1})
	\xrightarrow[\theta\to0^+]{}+\infty,
	\]
	since $p>2$ implies $2/p-1<0$. Therefore $\Phi(\theta)>0$ for all sufficiently small $\theta$,
	i.e.\ \eqref{eq:necessary-ineq-theta} fails for all $\theta\in(0,\theta_0)$ for some $\theta_0>0$.
\end{proof}

\begin{proof}[Proof of Theorem \ref{thm:main1}]
	Assume $p>2$. Consider the $2\times2$ matrix $Z_\theta$ defined in \eqref{eq:Ztheta}.
	If \eqref{eq:Bourin_Lee_p_le_2} held for all $Z\in\mathbb{M}_2$, then it would hold for $Z_\theta$ and imply the
	trace inequality \eqref{eq:trace-necessary}. Claims \ref{cla:left} and \ref{cla:right} reduce \eqref{eq:trace-necessary}
	to \eqref{eq:necessary-ineq-theta}, which fails for all sufficiently small $\theta$ by Claim \ref{cla:asymptotic}.
	Hence \eqref{eq:Bourin_Lee_p_le_2} cannot hold for all $Z\in\mathbb{M}_2$ when $p>2$.
	
	If $n>2$, embed $Z_\theta$ as $Z_\theta\oplus 0_{n-2}\in\Mn$; the same trace obstruction applies, so
	\eqref{eq:Bourin_Lee_p_le_2} fails in $\Mn$ as well. This contradicts the hypothesis, so necessarily $p\le2$.
\end{proof}

\section{Proof of Theorem~\ref{thm:main2}}
\label{sec:proof-main2}

\begin{proof}[Proof of Theorem~\ref{thm:main2}]
	Let $H=\ell^2(\mathbb N)$ with its standard orthonormal basis $(e_n)_{n\ge1}$, and define a finite-rank operator $Z\in\mathbb K$ by
	\[
	Ze_1=e_2,\qquad Ze_2=e_3,\qquad Ze_n=0\ \ (n\ge3).
	\]
	Equivalently, $Z$ is the unilateral shift truncated to the first three coordinates, that is,
	\[
	Z=
	\begin{pmatrix}
		0&0&0\\
		1&0&0\\
		0&1&0
	\end{pmatrix}
	\oplus 0 \oplus 0 \oplus \cdots .
	\]
	
	We first compute $|Z|$ and $|Z^*|$. Since
	\[
	Z^*Z=\diag(1,\,1,\,0,\,0,\,0,\dots),
	\qquad
	ZZ^*=\diag(0,\,1,\,1,\,0,\,0,\dots),
	\]
	we get
	\[
	|Z|=(Z^*Z)^{1/2}=\diag(1,\,1,\,0,\,0,\,0,\dots),
	\qquad
	|Z^*|=(ZZ^*)^{1/2}=\diag(0,\,1,\,1,\,0,\,0,\dots).
	\]
	Hence, for every $p>0$,
	\[
	\frac{|Z|^{p}+|Z^{*}|^{p}}{2}
	=
	\diag\!\left(\frac12,\,1,\,\frac12,\,0,\,0,\dots\right),
	\]
	and therefore
	\[
	\left(\frac{|Z|^{p}+|Z^{*}|^{p}}{2}\right)^{1/p}
	=
	\diag\!\left(2^{-1/p},\,1,\,2^{-1/p},\,0,\,0,\dots\right).
	\]
	In particular, the singular values of this operator satisfy
	\[
	\mu_1^\downarrow=1,\qquad \mu_2^\downarrow=2^{-1/p},\qquad \mu_3^\downarrow=2^{-1/p},\qquad \mu_m^\downarrow=0\ (m\ge4).
	\]
	
	Next, compute the real and imaginary parts:
	\[
	\Ree Z=\frac{Z+Z^*}{2}
	=
	\begin{pmatrix}
		0&\frac12&0\\[2pt]
		\frac12&0&\frac12\\[2pt]
		0&\frac12&0
	\end{pmatrix}
	\oplus 0 \oplus 0 \oplus \cdots,
	\qquad
	\Ima Z=\frac{Z-Z^*}{2i}.
	\]
	A direct computation shows that $\Ree Z$ has eigenvalues $0$ and $\pm 2^{-1/2}$, hence
	\[
	\mu^\downarrow(\Ree Z)=\left(2^{-1/2},\,2^{-1/2},\,0,\,0,\,0,\dots\right).
	\]
	Moreover, $\Ima Z$ is unitarily similar to $\Ree Z$.
	Indeed, for the $3\times 3$ leading block set $D=\diag(1, -i,-1)$ and extend it by
	$\oplus I$ on the remaining coordinates. Then one checks that
	\[
	\Ima Z = D\,(\Ree Z)\,D^*.
	\]
	Hence $\Ree Z$ and $\Ima Z$ have the same singular values.
	\[
	\mu^\downarrow(\Ima Z)=\left(2^{-1/2},\,2^{-1/2},\,0,\,0,\,0,\dots\right).
	\]
	Therefore,
	\[
	\mu_1^\downarrow(\Ima Z)=2^{-1/2},
	\qquad
	\mu_3^\downarrow(\Ree Z)=0.
	\]
	
	Choose $j=2$ and $k=0$, so that $1+j+k=3$, $1+j=3$, and $1+k=1$. Then for every $p>2$,
	\[
	\mu^{\downarrow}_{1+j+k}\!\left(\left(\frac{|Z|^{p}+|Z^{*}|^{p}}{2}\right)^{1/p}\right)
	=
	\mu_3^\downarrow\!\left(\left(\frac{|Z|^{p}+|Z^{*}|^{p}}{2}\right)^{1/p}\right)
	=
	2^{-1/p},
	\]
	while
	\[
	\mu^{\downarrow}_{1+j}(\Ree Z)
	+\mu^{\downarrow}_{1+k}(\Ima Z)
	=
	\mu_3^\downarrow(\Ree Z)+\mu_1^\downarrow(\Ima Z)
	=
	0+2^{-1/2}
	=
	2^{-1/2}.
	\]
	Since $p>2$ implies $2^{-1/p}>2^{-1/2}$, we conclude that
	\[
	\mu^{\downarrow}_{1+j+k}\!\left(\left(\frac{|Z|^{p}+|Z^{*}|^{p}}{2}\right)^{1/p}\right)
	>
	\mu^{\downarrow}_{1+j}(\Ree Z)
	+\mu^{\downarrow}_{1+k}(\Ima Z).
	\]
\end{proof}

\begin{rem}
	The construction is intentionally minimal and answers the \emph{existence} question in Question~\ref{ques:2}.
	Here $Z$ has rank $2$, and the operator $\left(\frac{|Z|^{p}+|Z^{*}|^{p}}{2}\right)^{1/p}$ is diagonal and easy to read off explicitly.
	The strict inequality is obtained by selecting indices $(j,k)=(2,0)$ for which the corresponding singular value of $\Ree Z$ vanishes at level $1+j=3$, while $\Ima Z$ contributes only its top singular value $2^{-1/2}$, yet the third singular value on the left-hand side equals $2^{-1/p}>2^{-1/2}$ whenever $p>2$.
\end{rem}
\section{Euler-type identities and isometry orbits}
\label{sec:euler-isometry}

In this section, we develop an operator-valued Euler identity and its consequences in the language of isometry orbits.
Our goal is twofold. First, we give a proof of Theorem~\ref{thm:main3} by encoding the Euler-type relation
into a unitary conjugation of suitable block matrices and then decomposing the resulting positive operator into
isometric compressions of its blocks. Second, we collect a convenient isometry decomposition principle for positive
block matrices, which will be used later to derive orbit-dominance statements from positivity and symmetry.

We start with a Pythagoras-type theorem for partitioned matrices due to Bourin and Lee~\cite{BLL24},
which allows us to express the square modulus of a block matrix as a sum of isometric conjugations of the square moduli
of its blocks.

\begin{lem}[{\cite[Theorem~2.1]{BLL24}}]\label{lem:pythagoras-partitioned}
	Let $m,n\ge1$ and let $T=[T_{ij}]_{i,j=1}^m$ be an $m\times m$ block matrix with $T_{ij}\in \Mn$.
	Then there exist isometries $W_{ij}\in \mathbb{M}_{mn,n}$ such that
	\[
	|T|^{2}=\sum_{i=1}^m\sum_{j=1}^m W_{ij}\,|T_{ij}|^{2}\,W_{ij}^*.
	\]
\end{lem}

\begin{rem}
	The standard $m\times m$ grid partitioning is row/column compatible in the sense of \cite{BLL24}, so the above lemma applies directly to block matrices $T=[T_{ij}]$ written with respect to such a partition.
\end{rem}
We now apply Lemma~\ref{lem:pythagoras-partitioned} to a block matrix obtained from a Hadamard conjugation.
This provides an isometric-orbit refinement of the Euler operator identity and yields Theorem~\ref{thm:main3}.
\begin{proof}[Proof of Theorem~\ref{thm:main3}]
	Let
	\[
	X:=
	\begin{pmatrix}
		A+B+C\\
		A\\
		B\\
		C
	\end{pmatrix},
	\qquad
	Y:=
	\begin{pmatrix}
		A+B\\
		B+C\\
		C+A\\
		0
	\end{pmatrix}
	\in \mathbb{M}_{4n,n}.
	\]
	Consider the $4\times 4$ Hadamard matrix
	\[
	H=\frac12
	\begin{pmatrix}
		1&1&1&1\\
		1&-1&1&-1\\
		1&1&-1&-1\\
		-1&1&1&-1
	\end{pmatrix},
	\qquad H^*H=I_4,
	\]
	and set $\mathcal U:=H\otimes I_n\in\mathbb{M}_{4n}$.
	A direct block computation gives $X=\mathcal U\,Y$.
	Hence $X^*X=Y^*Y$, i.e.\ \eqref{eq:Euler-op}.
	
	Let
	\[
	\Delta_Y:=
	\begin{pmatrix}
		A+B&0&0&0\\
		0&B+C&0&0\\
		0&0&C+A&0\\
		0&0&0&0
	\end{pmatrix}\in \mathbb{M}_{4n}.
	\]
	Set
	\[
	T:=\mathcal U\,\Delta_Y\,\mathcal U^*
	=\frac12
	\begin{pmatrix}
		A+B+C & A & B & C\\
		A & A+B+C & -C & -B\\
		B & -C & A+B+C & -A\\
		C & -B & -A & A+B+C
	\end{pmatrix}.
	\]
	Since $\mathcal U$ is unitary,
	\[
	|T|^2
	=|\mathcal U\Delta_Y\mathcal U^*|^2
	=\mathcal U\,|\Delta_Y|^2\,\mathcal U^*
	=\mathcal U\,
	\diag(|A+B|^2,\ |B+C|^2,\ |C+A|^2,\ 0)\,
	\mathcal U^*.
	\]
	Apply Lemma~\ref{lem:pythagoras-partitioned} with $m=4$ to the $4\times4$ block matrix $T=[T_{ij}]$.
	We obtain isometries $W_{ij}\in \mathbb{M}_{4n,n}$ such that
	\[
	|T|^2=\sum_{i=1}^4\sum_{j=1}^4 W_{ij}\,|T_{ij}|^{2}\,W_{ij}^*.
	\]
	Every block $T_{ij}$ equals $(A+B+C)/2$, $\pm A/2$, $\pm B/2$ or $\pm C/2$.
	Hence
	\[
	|T_{ij}|^2 \in \left\{ \tfrac14|A+B+C|^2,\ \tfrac14|A|^2,\ \tfrac14|B|^2,\ \tfrac14|C|^2\right\}.
	\]
	Moreover, each of $(A+B+C)/2,\ A/2,\ B/2,\ C/2$ appears exactly four times among the sixteen blocks.
	Grouping the corresponding four terms and renaming the isometries, we obtain isometries $V_{ij}\in\mathbb{M}_{4n,n}$ such that
	\[
	|T|^2
	=
	\frac14\sum_{j=1}^4V_{1j}\,|A+B+C|^2\,V_{1j}^*
	+\frac14\sum_{j=1}^4V_{2j}\,|A|^2\,V_{2j}^*
	+\frac14\sum_{j=1}^4V_{3j}\,|B|^2\,V_{3j}^*
	+
	\frac14\sum_{j=1}^4V_{4j}\,|C|^2\,V_{4j}^*.
	\]
	Conjugating by $\mathcal U^*$ yields
	\[
	|\Delta_Y|^2
	=
	\frac14\sum_{j=1}^4(\mathcal U^*V_{1j})\,|A+B+C|^2\,(\mathcal U^*V_{1j})^*
	+\cdots+
	\frac14\sum_{j=1}^4(\mathcal U^*V_{4j})\,|C|^2\,(\mathcal U^*V_{4j})^*.
	\]
	Finally, each $\mathcal U^*V_{ij}$ is still an isometry in $\mathbb{M}_{4n,n}$ since
	\[
	(\mathcal U^*V_{ij})^*(\mathcal U^*V_{ij})=V_{ij}^*\mathcal U\mathcal U^*V_{ij}=V_{ij}^*V_{ij}=I_n.
	\]
	Renaming $U_{ij}:=\mathcal U^*V_{ij}$ gives \eqref{eq:main3}.
\end{proof}

Next, we record a useful lemma of Bourin and Lee~\cite{BL12}, which decomposes a $2\times2$ positive block matrix
into a sum of two unitary orbits of its diagonal compressions. This lemma will serve as a building block for a
higher-dimensional isometry decomposition.

\begin{lem}[{\cite[Lemma~3.4]{BL12}}]\label{lem:BL12-3.4}
	Let
	$
	\begin{pmatrix}
		X & Y\\
		Y^{*} & Z
	\end{pmatrix}\in\mathbb{M}_{2n}
	$
	be positive semidefinite, written in $n\times n$ blocks. Then there exist unitary matrices
	$U,V_{0}\in\mathbb{M}_{2n}$ such that
	\[
	\begin{pmatrix}
		X & Y\\
		Y^{*} & Z
	\end{pmatrix}
	=
	U\begin{pmatrix}
		X & 0\\
		0 & 0
	\end{pmatrix}U^{*}
	+
	V_{0}\begin{pmatrix}
		0 & 0\\
		0 & Z
	\end{pmatrix}V_{0}^{*}.
	\]
\end{lem}
As an immediate consequence of Lemma~\ref{lem:BL12-3.4}, one obtains an isometry decomposition for positive block matrices
with an arbitrary number of blocks. For the reader's convenience, we include a self-contained proof.
\begin{lem}\label{lem:isometry-decomp}
	Let $m\ge 2$ and let $H=[H_{ij}]_{i,j=1}^m\in\mathbb{M}_{mn}$ be positive semidefinite,
	written in $n\times n$ blocks. Then there exist isometries
	$V_1,\dots,V_m\in\mathbb{M}_{mn,n}$ (i.e., $V_k^*V_k=I_n$) such that
	\begin{equation}\label{eq:isometry-decomp}
		H=\sum_{k=1}^m V_k\,H_{kk}\,V_k^*.
	\end{equation}
\end{lem}

\begin{proof}
	Since $H\ge0$, let $R:=H^{1/2}$, so $R=R^*$ and $H=R^2=RR^*=R^*R$.
	Partition $R$ into block columns $R=[R_1\ \cdots\ R_m]$.
	Then
	\begin{equation}\label{eq:HsumRk}
		H=RR^*=\sum_{k=1}^m R_kR_k^*.
	\end{equation}

	\medskip
	\noindent\emph{Step 1: Identify the diagonal blocks.}
	By block multiplication, the $(k,k)$ diagonal block of $H=R^*R$ is
	\[
	H_{kk}=R_k^*R_k\in\mathbb{M}_n^+.
	\]
	In particular, $|R_k|:=(R_k^*R_k)^{1/2}=H_{kk}^{1/2}$.
	
	\medskip
	\noindent\emph{Step 2: Polar decomposition and a partial isometry.}
	Take the polar decomposition of each $R_k$:
	\[
	R_k=U_k\,|R_k|=U_k\,H_{kk}^{1/2},
	\]
	where $U_k\in\mathbb{M}_{mn,n}$ is a partial isometry satisfying
	\[
	U_k^*U_k=P_k,
	\qquad
	P_k:=\mathrm{supp}(|R_k|)=\mathrm{supp}(H_{kk})
	\]
	(the orthogonal projection onto $\overline{\mathrm{Ran}}(|R_k|)$).
	Note that $H_{kk}^{1/2}(I_n-P_k)=0$.
	
	\medskip
	\noindent\emph{Step 3: Extend $U_k$ to a genuine isometry $V_k$.}
	If $P_k=I_n$, then $U_k^*U_k=I_n$ and $U_k$ is already an isometry; set $V_k:=U_k$.
	
	Assume now $P_k\neq I_n$. Set $d_k:=\mathrm{rank}(P_k)\le n$.
	Then $\dim\mathrm{Ran}(I_n-P_k)=n-d_k$.
	Also, since $U_k$ is a partial isometry with initial projection $P_k$,
	its range projection $Q_k:=U_kU_k^*$ has rank $d_k$, hence
	\[
	\dim\mathrm{Ran}(I_{mn}-Q_k)=mn-d_k\ \ge\ n-d_k,
	\]
	using $m\ge2\Rightarrow mn\ge n$.
	
	Therefore, there exists an isometry
	\[
	W_k:\mathrm{Ran}(I_n-P_k)\longrightarrow \mathrm{Ran}(I_{mn}-Q_k),
	\]
	i.e.\ a matrix $W_k\in\mathbb{M}_{mn,n}$ such that
	\[
	W_k^*W_k=I_n-P_k,
	\qquad
	W_kW_k^*\le I_{mn}-Q_k.
	\]
	(For example, choose orthonormal bases of the two subspaces and map one to the other.)
	
	Now define
	\[
	V_k:=U_k+W_k\in\mathbb{M}_{mn,n}.
	\]
	Then we claim the cross terms vanish. First, since $Q_k:=U_kU_k^*$ is the
	range projection of $U_k$, we have $Q_kU_k=U_k$, hence $U_k^*(I_{mn}-Q_k)=0$.
	On the other hand, $W_kW_k^*\le I_{mn}-Q_k$ implies $\Ran(W_k)\subseteq\Ran(I_{mn}-Q_k)$,
	equivalently $(I_{mn}-Q_k)W_k=W_k$. Therefore
	\[
	U_k^*W_k=U_k^*(I_{mn}-Q_k)W_k=0.
	\]
	Similarly, $U_k$ has initial projection $P_k=U_k^*U_k$, so $U_kP_k=U_k$ and hence
	$(I_n-P_k)U_k^*=0$. Since $W_k^*W_k=I_n-P_k$, we have $\Ran(W_k^*)\subseteq\Ran(I_n-P_k)$,
	equivalently $W_k^*=W_k^*(I_n-P_k)$. Thus
	\[
	W_k^*U_k=W_k^*(I_n-P_k)U_k=0.
	\]
	Consequently,
	\[
	V_k^*V_k=(U_k+W_k)^*(U_k+W_k)=U_k^*U_k+W_k^*W_k
	=P_k+(I_n-P_k)=I_n,
	\]
	so $V_k$ is an isometry in $\mathbb{M}_{mn,n}$.
	
	Moreover, since $H_{kk}^{1/2}(I_n-P_k)=0$, we still have
	\[
	V_k H_{kk}^{1/2}
	=
	(U_k+W_k)H_{kk}^{1/2}
	=
	U_kH_{kk}^{1/2}
	=
	R_k.
	\]
	
	\medskip
	\noindent\emph{Step 4: Conclude the decomposition.}
	From $R_k=V_kH_{kk}^{1/2}$ we deduce
	\[
	R_kR_k^*
	=
	V_kH_{kk}^{1/2}\,H_{kk}^{1/2}V_k^*
	=
	V_kH_{kk}V_k^*.
	\]
	Summing over $k$ and using \eqref{eq:HsumRk} gives
	\[
	H=\sum_{k=1}^m R_kR_k^*=\sum_{k=1}^m V_kH_{kk}V_k^*,
	\]
	which is exactly \eqref{eq:isometry-decomp}.
\end{proof}
We are now ready to prove Theorem~\ref{thm:first} and Theorem~\ref{thm:second}.
In both cases, we conjugate an appropriate direct sum by a Fourier matrix $F_m\otimes I_n$ to obtain a positive block matrix
with constant diagonal blocks, and then apply Lemma~\ref{lem:isometry-decomp} to conclude the desired isometry-orbit domination.

\begin{proof}[Proof of Theorem~\ref{thm:first}]
	First, we compute Fourier conjugations of two direct sums.
	Let $\omega=e^{2\pi i/3}$ and set
	\[
	W_3=\frac1{\sqrt3}
	\begin{pmatrix}
		I & I & I\\
		I & \omega I & \omega^2 I\\
		I & \omega^2 I & \omega I
	\end{pmatrix}
	=F_3\otimes I_n\in\mathbb{M}_{3n}.
	\]
	For
	\[
	D:=|A+B|^2\oplus|B+C|^2\oplus|A+C|^2=\diag(D_1,D_2,D_3),
	\quad
	\begin{cases}
		D_1=|A+B|^2,\\
		D_2=|B+C|^2,\\
		D_3=|A+C|^2,
	\end{cases}
	\]
	one has
	\begin{equation}\label{eq:W3DW3}
		W_3DW_3^*=\frac13
		\begin{pmatrix}
			D_1+D_2+D_3 & D_1+\omega^2 D_2+\omega D_3 & D_1+\omega D_2+\omega^2 D_3\\
			D_1+\omega D_2+\omega^2 D_3 & D_1+D_2+D_3 & D_1+\omega^2 D_2+\omega D_3\\
			D_1+\omega^2 D_2+\omega D_3 & D_1+\omega D_2+\omega^2 D_3 & D_1+D_2+D_3
		\end{pmatrix}.
	\end{equation}
	In particular, every diagonal block of $W_3DW_3^*$ equals $(D_1+D_2+D_3)/3$.
	
	Clearly, $W_3DW_3^*\ge 0$ and, by \eqref{eq:W3DW3}, each diagonal block equals $(D_1+D_2+D_3)/3$.
	A direct expansion shows
	\[
	D_1+D_2+D_3
	=
	|A+B|^2+|B+C|^2+|A+C|^2
	=
	|A+B+C|^2+|A|^2+|B|^2+|C|^2.
	\]
	Denote this common sum by $T$. Hence $W_3DW_3^*$ is a $3\times3$ block positive matrix with constant diagonal block $T/3$.
	Applying \eqref{eq:isometry-decomp} to $H=W_3DW_3^*$ yields isometries $V_1,V_2,V_3\in\mathbb{M}_{3n,n}$ such that
	\[
	W_3DW_3^*=\sum_{k=1}^3 V_k\,(T/3)\,V_k^*.
	\]
	Conjugating by $W_3^*$ and setting $U_k=W_3^*V_k$ (still isometries) gives the claim.
\end{proof}

\begin{proof}[Proof of Theorem~\ref{thm:second}]	First, let $\eta=e^{2\pi i/4}=i$ and set
	\[
	W_4=\frac1{2}
	\begin{pmatrix}
		I & I & I & I\\
		I & \eta I & \eta^2 I & \eta^3 I\\
		I & \eta^2 I & \eta^4 I & \eta^6 I\\
		I & \eta^3 I & \eta^6 I & \eta^9 I
	\end{pmatrix}
	=F_4\otimes I_n\in\mathbb{M}_{4n},
	\]
	so that $\eta^2=-1$, $\eta^3=-i$, $\eta^4=1$. For
	\[
	E:=|A+B+C|^2\oplus|A|^2\oplus|B|^2\oplus|C|^2=\diag(E_1,E_2,E_3,E_4),
	\]
	with $E_1=|A+B+C|^2$, $E_2=|A|^2$, $E_3=|B|^2$, $E_4=|C|^2$, we obtain
	\begin{equation}\label{eq:W4EW4}
		W_4EW_4^*=\frac14\bigl[M_{rs}\bigr]_{r,s=1}^4,
		\qquad
		M_{rs}=E_1+\eta^{-(r-s)}E_2+\eta^{-2(r-s)}E_3+\eta^{-3(r-s)}E_4.
	\end{equation}
	In particular, every diagonal block of $W_4EW_4^*$ equals $(E_1+E_2+E_3+E_4)/4$.
	
	Clearly, $W_4EW_4^*\ge 0$ and, by \eqref{eq:W4EW4}, each diagonal block equals $(E_1+E_2+E_3+E_4)/4$.
	As in the previous proof,
	\[
	E_1+E_2+E_3+E_4
	=
	|A+B+C|^2+|A|^2+|B|^2+|C|^2
	=
	|A+B|^2+|B+C|^2+|A+C|^2.
	\]
	Denote this sum by $S$. Thus $W_4EW_4^*$ is a $4\times4$ block positive matrix with constant diagonal block $S/4$.
	Applying \eqref{eq:isometry-decomp} to $H=W_4EW_4^*$ yields isometries $V_1,\dots,V_4\in\mathbb{M}_{4n,n}$ such that
	\[
	W_4EW_4^*=\sum_{k=1}^4 V_k\,(S/4)\,V_k^*.
	\]
	Conjugating by $W_4^*$ and setting $U_k=W_4^*V_k$ completes the proof.
\end{proof}
\begin{rem}
	In the proof of Theorem~\ref{thm:second} above, the choice of the unitarily conjugating matrix is not unique. For instance, we could pick a Hadamard matrix 
	\[
	H_4=\frac12
	\begin{pmatrix}
		1 & 1 & 1 & 1\\
		1 & 1 & -1 & -1\\
		1 & -1 & 1 & -1\\
		1 & -1 & -1 & 1
	\end{pmatrix}
	\] instead of the $4\times 4$ Fourier transform $F_4$: in such a case, dealing with real matrices yields
	real isometries.
\end{rem}

%	\section*{Declaration of competing interest}
%	The author declares no competing interests.
%	
%	\section*{Data availability}
%	No data was used for the research described in the article.
%	
%	\section*{Acknowledgments}
%	Teng Zhang is supported by the China Scholarship Council, the Young Elite Scientists Sponsorship Program for PhD Students (China Association for Science and Technology), and the Fundamental Research Funds for the Central Universities at Xi'an Jiaotong University (Grant No.~xzy022024045).

\end{document}